\documentclass{article}
\usepackage[T2A]{fontenc}
\usepackage[utf8]{inputenc}
\usepackage[english]{babel}
\usepackage[tbtags]{amsmath}
\usepackage{amsfonts,amssymb,mathrsfs,amscd}
\usepackage{graphicx}

\usepackage{float}

\numberwithin{equation}{section}
\newtheorem{theorem}{Theorem}
\newtheorem{lemma}{Lemma}[section]
\newtheorem{cor}{Corollary}
\newtheorem{exmpl}{Example}

\newtheorem{proof}{Proof}
\newtheorem{remark}{Remark}

\begin{document}

\begin{center}
{\bf On uniqueness sets and coefficients of multiple Walsh series converging over cubes}
\end{center}

\begin{center}
A.~D.~Kazakova
\end{center}

We study problems on uniqueness sets ($U$-sets) for multiple Walsh series converging over cubes and the properties of the coefficients of such series. New broad classes of $U$-sets are constructed. In particular, it is proved that hyperplanes parallel to the coordinate ones are $U$-sets. For the coefficients of multiple Walsh series converging over cubes, both the index sets on which they can be made arbitrarily large and the index sets on which these coefficients tend to zero are described.

Bibliography: 49 titles.

Keywords: Walsh system, uniqueness sets, multiple series, convergence over cubes, Cantor–Lebesgue theorem

\section{Introduction}\label{s1}
The theory of uniqueness sets is one of the classical branches of the theory of orthogonal series. Uniqueness sets (otherwise, $U$-sets) are exceptional sets outside which convergence to zero of series in an orthogonal system does not violate uniqueness. Sets that are not $U$-sets are called $M$-sets.

Research by many authors \cite{Zygmynd}, \cite{bari-1961}, in particular the Salem--Zygmund--Bari--Pyateckii-Shapiro theorem, has shown that the question of how to distinguish $U$- and $M$-sets is very subtle, related not only to the metric and topological, but also to the arithmetic structure of the sets. Thus, in \cite{Kechris-Louveau-1987} it is proved that there is no constructive criterion allowing one to recognize $U$-sets (even in the class of closed sets).

Some recent results concerning uniqueness problems, including those for other function systems, can be found in \cite{kholshchevnikova-2019, kozma-olevskii-2020, plotnikov-2020, skvortsov-2021, wronicz-2021, lukomskii-2021, plotnikov-2022, gevorkyan-2023, gevorkyan-2024, Gev25, keryan-khachatryan-2024, kazakova-plotnikov-2025-1, kazakova-plotnikov-2025-2, kazakova-plotnikov-2025-3}.

In the multidimensional case, the uniqueness theory of orthogonal series turns out to be much more complicated than in the one-dimensional case. There are not many results here, and many fundamental questions remain open. Thus, for multiple trigonometric series, it is still unknown whether even the empty set is a $U$-set for convergence over cubes. The hypothesis of J.M.~Ash \cite{Ash-2007} is that it is not. In other words, J.M.~Ash suggests that there could exist a multiple trigonometric series, not all of whose coefficients are zero, converging everywhere to zero over cubes. A significant breakthrough is the result of J. Bourgain \cite{Bourgain-1996}, who showed that the empty set is a $U$-set for convergence over spheres for multiple trigonometric series. In the case of convergence over rectangles, a breakthrough is the work of Sh.T.~Tetunashvili \cite{tetynashvili-1993}, in which a broad class of $U$-sets is constructed, including all countable sets. A weaker result, stating that the empty set is a $U$-set for double trigonometric series with convergence over rectangles, was proved earlier by J.M.~Ash and G. Welland \cite{Ash-Welland-1972}. Other broad classes of continuum $U$-sets are contained in the works of L.D.~Gogoladze \cite{Gogoladze-2008} and T.A.~Zherebyeva \cite{Zhe10} (see also the work \cite{ash-freiling-rinne-1993}).

In the present paper, we consider uniqueness questions for another classical system of functions --- the Walsh system.
For multiple Walsh series, the first results on uniqueness sets appeared in the works of H.O.~Movsisyan \cite{Movisyn} and V.A.~Skvortsov \cite{Scv-1975}. They independently proved that any at most countable set is a $U$-set for convergence over rectangles. Moreover, from the results of \cite{Scv-1975} it follows that any union of a countable number of hyperplanes is a $U$-set for convergence over rectangles. Significant progress in constructing new classes of continuum uniqueness sets for convergence over rectangles was made by S.F.~Lukomskii \cite{Lukomskii-1989} and later by L.D.~Gogoladze \cite{Gogoladze-2008} and T.A.~Zherebyeva \cite{Zherebeva-2009}.
See also the work of V.A.~Skvortsov and F.~Tulone \cite{skvortsov-tulone-2015}, where such classes are also essentially contained, and the work of N.N.~Kholshchevnikova \cite{Kho02}, which investigates questions about the union of uniqueness sets, which also allows obtaining new classes of $U$-sets. In \cite{Lukomskii-1989}, the following theorem was proved:

Theorem A (Theorem 3, \cite{Lukomskii-1989}) Let $\phi(t^2, \ldots, t^d)$ be an at most countably-valued function. Then the surface in $[0,1]^d$ given by the equation $t^1 = \phi(t^2, \ldots, t^d)$ is a uniqueness set for the $d$-dimensional Walsh system.

In the case of convergence over cubes, which is weaker than convergence over rectangles, the situation becomes much more complicated. However, unlike the trigonometric system, the uniqueness theory for the multiple Walsh system under convergence over cubes has seen some development. It is worth noting that the natural domain of definition for the Walsh system is the dyadic group $\mathbb{G}$, for which the Walsh system is a character system.

In \cite{luk-dis-1996}, S.F.~Lukomskii showed that the empty set is a uniqueness set for convergence over cubes when considering Walsh functions on the dyadic group $\mathbb{G}$. If one considers $d$-dimensional Walsh functions on the unit cube $[0,1]^d$, then, as in the case of $d$-multiple trigonometric series, it is still unknown whether the empty set is a $U$-set.

A significant breakthrough in the uniqueness questions for multiple Walsh series in the case of convergence over cubes and $\lambda$-convergence was made by M.G.~Plotnikov \cite{plotnikov-07-1, plotnikov-07-2, plotnikov-2008, plotnikov-2010, plotnikov-2017, MP-EMJ-2019}; in particular, it was proved that any countable set is a uniqueness set for convergence over cubes. As can be seen from these works, studying uniqueness questions for convergence over cubes requires a detailed analysis of the arithmetic structure of the set and the application of more subtle methods than for convergence over rectangles. In \cite{plotnikov-2017}, the following result was obtained:

Theorem B (Theorem 3, \cite{plotnikov-2017}) Any finite union of planes parallel to the coordinate ones is a uniqueness set for $2$-convergence for multiple Walsh series.

Later, S.F.~Lukomskii \cite{lukomskii-2021} considered a narrower class of sets but with weaker convergence than in Theorem B. The following theorem was proved:

Theorem C (Theorem 1, \cite{lukomskii-2021}) Segments of the form $\boldsymbol{\xi} \times \mathbb{G}$, where $\boldsymbol{\xi} \in \mathbb{G}^{d-1}$, are uniqueness sets for multiple Walsh series under convergence over cubes.

For the one-dimensional Walsh system $\{W_n(g)\}$, a set $E \subset \mathbb{G}$ is called a Dirichlet set (see \cite[Chapter 7]{Schipp-and-Co}) if $\liminf_{n \to \infty} \sup_{g \in E} |1-W_n(g)| = 0$. In uniqueness theory, Dirichlet sets were studied by K.~Yoneda~\cite{yoneda-1984, yoneda-1982}, who established in particular~\cite{yoneda-1982} that all one-dimensional Dirichlet sets are $U$-sets for the Walsh system. Following the definition of Dirichlet sets in the one-dimensional case, we call a set $E \subset \mathbb{G}^d$ a Dirichlet set for the $d$-dimensional Walsh system if there exists a $d$-dimensional sequence $\mathbf{N} = ( \mathbf{N}_i \in \mathbb{N}^d, \, i \in \mathbb{N} )$ such that
\begin{equation} 
\label{Eq:U-05} 
E := WD^d(\mathbf{N}) := \bigcap_{i\in \mathbb{N}} WD^d(\mathbf{N}_i); \quad WD^d(\mathbf{N}_i) = \bigcap_{i\in \mathbb{N}} \{ \mathbf{g} \in \mathbb{G}^d \colon W_{ \mathbf{N}_i} ( \mathbf{g} ) = 1 \}. 
\end{equation}
As in the one-dimensional case (see \cite{yoneda-1982}), Dirichlet sets $WD^d(\mathbf{N})$ are subgroups of the group $\mathbb{G}^d$.

In \cite{plotnikov-2010}, M.~G.~Plotnikov studied Dirichlet sets of the form \( WD^d(2^{K}\mathbf{1}) \) (see \S \ref{sp-il}, Fig.~\ref{fig:2}a). Here \( 2^{K}\mathbf{1} \) denotes the sequence \( (2^{k_i}\mathbf{1}, \, i \in \mathbb{N}) \) constructed from a given sequence \( K = (k_i \in \mathbb{N}_0,  \, i\in \mathbb{N}) \). The following result was proved.

Theorem D (Theorem 11, \cite{plotnikov-2010}) The set $WD^d(2^{K}\mathbf{1})$ is a $U$-set for the $d$-dimensional Walsh system under convergence over cubes.

As a corollary of Theorem D, Dirichlet sets of the form $WD^d(2^{K}\mathbf{1})$ are uniqueness sets for convergence over rectangles. On the other hand, Theorem A allows us to specify another class of Dirichlet sets that are $U$-sets for convergence over rectangles. Consider the two-dimensional case and let an increasing sequence $(N_i \in \mathbb{N}_0,  \, i\in \mathbb{N})$ be given. Set \begin{equation} \label{spec-N}
 \mathbf{N} = (\mathbf{N}_i \in \mathbb{N}^2, \, i \in \mathbb{N} ) , \qquad  \mathbf{N}_i=(2^{i-1}, N_i).
\end{equation} Then $WD^2 (\mathbf{N})$ is a $U$-set for convergence over rectangles, because in this case the mapping $\phi$ defined in Theorem A is given as $\phi(g^2) = \sum\limits_{i \in \mathbb{N}} \frac{1 - W_{N_i}(g^2)}{2^{i+1}}$ and is a single-valued function.

In the case where $N_i$ in formula \eqref{spec-N} is such that $2^{m_i-i} \le N_i < 2^{2m_i-i}$ and $\frac{m_{i+1}}{m_i} \ge 2$, a two-dimensional Lukomskii set can be constructed from Dirichlet sets:
\begin{equation} \label{luk}
    E^2 := \bigcap_{i\in \mathbb{N}} \bigcup_{N_i = 2^{m_i-i}}^{2^{2m_i-i}-1} WD^2( \mathbf{N}_i) \cap P( \mathbf{N}_i); 
\end{equation}
here $P( \mathbf{N}_i)$ are rectangles of the form $\left[\frac{j}{2^{i-1}}, \; \frac{j+1}{2^{i-1}}\right]\times\left[\frac{k}{2^{m_i-i+1}}, \; \frac{k+1}{2^{m_i-i+1}}\right]$, where each $N_i$ uniquely corresponds to a specially chosen pair of indices $(j, k)$.
In \cite{Lukomskii-1992}, S.F.~Lukomskii showed that there exists a set $E^d \subset\mathbb{G}^d$ which is a $U$-set for convergence over rectangles, but an $M$-set for convergence over cubes and $\lambda$-convergence. In the two-dimensional case, as we noted (see \cite{Kazakova-2026}), the Lukomskii set $E^2$ can be described by formula \eqref{luk} (see \S \ref{sp-il}, Fig.\ref{fig:1}a).

An $M$-set $F^d \subset \mathbb{G}^d$ (see \S \ref{sp-il}, Fig.\ref{fig:1}b) for repeated convergence, convergence over rectangles and, as a consequence, over cubes and $\lambda$-convergence, was constructed in \cite{kazakova-plotnikov-2025-3}. There, a visual description of the set $F^d$ is given, and a modification of its construction is proposed, leading to Dirichlet sets $WD^d (\mathbf{N})$, which under certain constraints on the sequence $\mathbf{N}$ are $U$-sets for $\lambda$-convergence and, consequently, over rectangles.

Note that the $M$-sets $E^d$ and $F^d$ have a rather nontrivial structure. A natural question arises: will sufficiently simple in a certain sense sets (of Dirichlet type, planes, etc.) be uniqueness sets for convergence over cubes? In this paper (Sec. \ref{p-2-3}), progress is made on this question. In Theorems \ref{th-3-1}, \ref{th-3-5}, \ref{th-3-2}, \ref{th-3-4}, \ref{th-3-3}, \ref{th-3-7}, new classes of $U$-sets for convergence over cubes are constructed. Theorems \ref{th-3-1} and \ref{th-3-5} (see also Corollaries \ref{important-WD} and \ref{no-important}) establish that certain sets, whose construction uses multidimensional Dirichlet sets, are $U$-sets for convergence over cubes. In Section \ref{planes}, dyadic planes are introduced and examples are given. These planes include, in particular, diagonal and inclined planes passing through dyadic rational points, as well as planes parallel to the coordinate ones. Theorems \ref{th-3-2}, \ref{th-3-4} and \ref{th-3-3} (see also Corollary \ref{important}) prove that dyadic planes are $U$-sets for convergence over cubes. This significantly expands the existing knowledge about uniqueness sets, since until now not a single example of a hyperplane (for $d>2$) that is a $U$-set for convergence over cubes was known. Recall that Theorem B considered stronger convergence, while Theorem C only considered segments.

Theorem \ref{th-3-3} (see also Corollary \ref{important}) generalizes both Theorems B and C at once and provides an answer to the question posed by S.F.~Lukomskii in \cite{lukomskii-2021}: let $d > 2$, $d = d_1 + d_2$, $\mathbf{x} \in \mathbb{G}^{d_2}$; for which $d_1$ and $d_2$ is the set $\mathbb{G}^{d_1} \times \{\mathbf{x}\}$ a uniqueness set for convergence over cubes? We obtain that for all admissible $d_1$ and $d_2$ ($d_1<d$), the set $\mathbb{G}^{d_1} \times \{\mathbf{x}\}$ is a $U$-set for convergence over cubes.

The content of this paper is as follows. \S~\ref{S:Prelim-1} contains the main definitions and auxiliary facts. Illustrations for the text are contained in \S~\ref{sp-il}. The main results are contained in \S~\ref{S:Main-results}.

In Section \ref{p-2-2}, Theorems 3 and 4 are proved, connecting the convergence of certain cubic sums of a series with properties that can be interpreted as a certain kind of continuity of the corresponding finitely additive set functions of dyadic type (so-called quasimeasures). More on quasimeasures in Sec. \ref{kvazi-mese}. These results are key for Section \ref{p-2-3}, the content of which was briefly outlined above. Note that the continuity of a quasimeasure in the classical sense, analogous to the one-dimensional case, is a tool for solving uniqueness problems for the case of convergence over rectangles \cite{skvortsov-tulone-2015} and does not work in the case of convergence over cubes.
As pointed out in \cite{MP-EMJ-2019}, this fundamental difference is largely due to the fact that the coefficients of convergent $d$-dimensional Walsh series behave significantly differently depending on whether we are dealing with convergence over all rectangles or over rectangles with a bounded aspect ratio. Namely, if a series converges over rectangles to a finite sum even on a special set of measure zero, then its coefficients tend to zero (see \cite{Scv-1975}). On the other hand, there exists a multiple Walsh series that converges over cubes even everywhere to a finite sum, but its coefficients from some massive set grow faster than any predetermined sequence (see \cite{plotnikov-2012}). In \cite{plotnikov-2012}, both index sets whose coefficients can be made arbitrarily rapidly growing, and index sets on which coefficients tend to zero, are identified.

In Section \ref{p-2-1}, new index sets of both types described above are constructed. Namely, a Cantor–Lebesgue type theorem for the diagonal coefficients of a series under certain conditions on their indices is proved (Theorem \ref{theorem-1-1}), and at the same time, Theorem \ref{theorem-1-2} contains an example of a series converging everywhere over cubes to a finite sum. Some of its diagonal coefficients with indices that do not satisfy the conditions of Theorem \ref{theorem-1-1} grow faster than any predetermined sequence. Note that, by virtue of P. Cohen's theorem \cite{Cohen-1958} and the theorem of J.M.~Ash and G.~Wang \cite{ash-Wang-1997}, such behavior is impossible for multiple trigonometric series converging over cubes on a set of full measure, whose coefficients grow slower than any exponential.

The work uses ideas from the works of M.G.~Plotnikov \cite{plotnikov-07-1}, \cite{plotnikov-2010} and \cite{plotnikov-2012}.

\section{Notation, main definitions, and auxiliary facts}
\label{S:Prelim-1}

\subsection{Notation}

We write $:=$ for equality by definition.

$\mathbb{C}$ is the set of complex numbers,
$\mathbb{N}$ is the set of natural numbers;
$\mathbb{N}_0 := \mathbb{N} \cup \{ 0 \}$.

\[
\mathrm{I} (A)
:=
\begin{cases}
1,  &\text{if the statement $A$ is true},
\\
0,  &\text{if the statement $A$ is false},
\end{cases}
\]
We write $n_k$ for the \emph{dyadic coefficients} of a number $n \in \mathbb{N}_0$,
which are taken from the \emph{dyadic expansion}
$n = \sum\nolimits_{k=0}^\infty n_k 2^k$, $n_k = 0 \vee 1$,
of the number $n$.

We use the notation $a : b$ for the set $\{ a, a+1, \ldots, b-1, b \}$.

Multiplication of a vector by a scalar is understood in the usual sense.

Throughout the paper, we fix a natural number $d \ge 2$.

We write $\mathbf{0}$ for the $d$-dimensional vector $(0, \ldots, 0)$,
$\mathbf{1}$ for the $d$-dimensional vector $(1, \ldots, 1)$,
$\mathbf{g} = ( g^1, \ldots, g^d )$, $\mathbf{h} = ( h^1, \ldots, h^d )$, ...
The notation $\mathbf{g} < \mathbf{h}$ ($\mathbf{g} \le \mathbf{h}$) means
that $g^j < h^j$ ($g^j \le h^j$) for all $j \in 1:d$.

By $\Sigma^{d}$ we denote the set of vectors $\boldsymbol{\sigma} = (\sigma_1, \dots, \sigma_d)$ such that $\sigma_q \in \{0, 1\}$ for all $q = 1:d$. The expression $|\boldsymbol{\sigma}|$ denotes the sum $\sigma_1 + \cdots + \sigma_d$.

Let $\Sigma^{d}_2:=\Sigma^d/\mathbb{Z}_2$ be the set of vectors $\boldsymbol{\sigma} \in \Sigma^d$ such that $|\boldsymbol{\sigma}| \equiv 0 \pmod{2}$.

By $\# N$ we denote the number of nonzero dyadic coefficients of the number $N \in \mathbb{N}_0$.

By $|J|$ we denote the number of elements in the set $J$.

\subsection{Main definitions}
\subsubsection{}
The \emph{dyadic group} $\mathbb{G} = \mathbb{G}_2$
is defined as the direct product of a countable set
of cyclic groups $\mathbb{Z}_2$ with the discrete topology,
endowed with the Tikhonov topology.
The zero of the group $\mathbb{G}$ and the inverse elements are defined
in the obvious way.
It is convenient to represent elements of $\mathbb{G}$ as sums
of formal series converging in the topology of $\mathbb{G}$:
\begin{equation}
\label{Eq:PG-El}
\bigoplus\limits_{ k = 0 }^\infty
g_k e_k,
\quad
g_k \in \{ 0, 1 \},
\end{equation}
or as the series themselves.
Here $e_k$ are the generating elements of the group $\mathbb{G}$, $2 e_k = 0$,
and the group operation $\oplus$ is applied to the elements \eqref{Eq:PG-El} componentwise.

For each $d \in \mathbb{N}$, the set $\mathbb{G}^d = ( \mathbb{G}_2 )^d$
is a topological abelian group
with addition $\mathbf{g} \oplus \mathbf{h}
:= ( g^1 \oplus h^1, \ldots, g^d \oplus h^d )$,
and the zero of the group and inverse elements
are defined naturally.
We denote by $\oplus$
the group operation both on the group $\mathbb{G}$
and on $\mathbb{G}^d$, and this will not lead to confusion.

A basis for the topology of the group $\mathbb{G}$ is formed by cosets
of the subgroups $\mathbb{G}_k := \left\{ \bigoplus\nolimits_{ t = k+1 }^\infty
g_t e_t \right\}$, which are called {\it dyadic intervals} of rank $k$ and
are often enumerated as follows:
\[
\Delta^{(k)}_m
:=
\left\{
    \bigoplus\nolimits_{ t = 0 }^\infty
    g_t e_t
    \colon
    g_t = m_{k - 1 - t}, \;\; t \in [ 0, k )
\right\},
\quad
k \in \mathbb{N}_0,
\quad
m \in 0 : 2^k - 1,
\]
where $m_t$ are the dyadic coefficients of the number $m$.
A basis for the topology of $\mathbb{G}^d$ is formed by $d$-dimensional
{\it dyadic cubes} (of rank $k$):
\begin{equation}
\label{Eq:P-adic-Cube}
\Delta^{(k)}_{\mathbf{m}}
:=
\Delta^{(k)}_{m^1}
\times
\ldots
\times
\Delta^{(k)}_{m^d},
\qquad
\mathbf{m} \in ( 0 : 2^k - 1 )^d,
\end{equation}
each of which is simultaneously open and closed in the topology of $\mathbb{G}^d$.
For some dyadic cube of rank $k$, we write $\Delta^{(k)}$. By $\Delta^{(k)}(\mathbf{g})$ we denote the smallest cube of rank $k$ containing the point $\mathbf{g}$.

The mapping $F$ that assigns to
the element \eqref{Eq:PG-El} the sum of the series
$\sum\limits_{k=0}^\infty \frac{g_k}{2^{k+1}}$
maps $\mathbb{G}$ bijectively (up to a countable set)
onto $[0,1]$, and $\Delta^{(k)}_m$ onto the sets
$\left[ m p^{-k}, \, ( m + 1 ) p^{-k} \right] \subset [0,1]$. Clearly, the preimages of each dyadic irrational point, zero, and one consist of a single element, while the preimages of each dyadic rational point $t \in [0,1]$, except zero and one, consist of two elements, which we denote by $t_- \in \mathbb{G}$ and $t_+ \in \mathbb{G}$. Here $t_-$ denotes the element of the group $\mathbb{G}$ with infinitely many $1$'s and $t_+$ denotes the element of $\mathbb{G}$ with infinitely many $0$'s.

The mapping \begin{equation} \label{otobr-F}
    ( g^1 , \ldots, g^d ) \stackrel{F}{\mapsto} ( F ( g^1 ), \ldots, F ( g^d ) )
\end{equation}
maps $\mathbb{G}^d$ bijectively (up to a set of measure zero) onto $[0,1]^d$,
and the dyadic cubes \eqref{Eq:P-adic-Cube}
onto the cubes $\mathop{\times}\limits_{ l = 1 }^d
\left[ m^l p^{-k}, \, ( m^l + 1 ) p^{-k} \right]$.

By the \emph{measure} $\mu_d$ on the group $\mathbb{G}^d$ we mean the normalized ($\mu_d ( \mathbb{G}^d ) = 1$)
Haar measure, defined on all Borel subsets of the group $\mathbb{G}^d$
and invariant under shifts and transformations mapping $H$ to $H^{-1}$.
In this case, $\mu_d ( \Delta^{(k)}_{\mathbf{m}} ) = 2^{-kd}$.
For the one-dimensional case, see \cite{Schipp-and-Co}. Sometimes we will omit the index of the measure indicating the dimension.

\subsubsection{}

On the group $\mathbb{G}$, the {\it Walsh functions} in Paley enumeration
are defined as
$W_n (g) = \prod_{k=0}^{\infty}
( - 1 )^{g_k n_k}$, where $n \in \mathbb{N}_0$,
and $g$ is an element of $\mathbb{G}$ of the form \eqref{Eq:PG-El}.

The \emph{$d$-dimensional Walsh functions}
$W_{\mathbf{n}}$,
\begin{equation}
\label{Eq:VC-F}
W_{\mathbf{n}} ( \mathbf{g} )
:=
\prod\limits_{l=1}^d
W_{n^l} (g^l),
\quad
\mathbf{n} \in ( \mathbb{N}_0 )^d,
\quad
\mathbf{g} \in \mathbb{G}^d,
\end{equation}
form an orthonormal system in $L^2 ( \mathbb{G}^d, \mu_d )$, and
\[
W_{ \mathbf{n} } ( \mathbf{g} )
W_{ \mathbf{n} } ( \mathbf{h} )
=
W_{\mathbf{n}} ( \mathbf{g} \oplus \mathbf{h} )
\quad
\text{for all $\mathbf{n}$, $\mathbf{g}$, $\mathbf{h}$}.
\]


A {\it $d$-dimensional Walsh series} on $\mathbb{G}^d$ is defined by the formula
\begin{equation}
 \label{series}
\sum_{ \mathbf{n} \in \mathbb{N}_0^d }
c_{ \mathbf{n} } W_{ \mathbf{n} } ( \mathbf{g} ),
\quad
c_{ \mathbf{n} } \in \mathbb{C},
\end{equation}
and its $\mathbf{N}$-th \emph{rectangular partial sums} at a point $\mathbf{g}$ are
\begin{equation}
\label{Eq:PartSums}
S_{\mathbf{N}} ( \mathbf{g} )
:=
\sum_{ \mathbf{n} < \mathbf{N} }
c_{ \mathbf{n} }
W_{\mathbf{n}} (\mathbf{g}),
\quad
\mathbf{N} \in \mathbb{N}^d.
\end{equation}
Partial sums $S_{\mathbf{N}}$ with indices $\mathbf{N} = N \mathbf{1}$
are called \emph{cubic} and are denoted simply by $S_N$.

For $\mathbf{n}$, $\mathbf{N} - \mathbf{1} < 2^k \mathbf{1}$,
each function $W_{\mathbf{n}}$ and the partial sum
$S_{\mathbf{N}}$ take constant values on $\Delta^{(k)}$
$=: W_{\mathbf{n}} ( \Delta^{(k)} )$ and
$=: S_{\mathbf{N}} ( \Delta^{(k)} )$, respectively.

The series \eqref{series} \emph{converges over rectangles} at a point $\mathbf{g}$
to a sum $S \in \mathbb{C}$ if
\[
\lim S_{\mathbf{N}} ( \mathbf{g} ) = S
\;\; \text{as $\min \{ N^1, \ldots, N^d \} \to \infty$},
\]
and \emph{over cubes} if $\lim\limits_{N \to \infty} S_N ( \mathbf{g} ) = S$.


\subsubsection{}

The function
\[
D_N (g)
:=
\sum_{ n < N }
W_n (g)
\]
is called the $N$-th {\it Dirichlet kernel} for the (one-dimensional) Walsh system.
The following formulas are known (see \cite{golubov-efimov-skvortsov}, formulas (1.4.9), (1.4.11), (1.4.13) and \cite{plotnikov-2012}, formula (3.9), Lemma 5):
\begin{equation}
\label{Eq:WFS}
D_n(x) = n,
\quad 1 \le n \le 2^{k+1}, \quad
x \in \Delta^{(k+1)}_0,
\end{equation}

\begin{equation}
\label{Eq:WFS-2}
D_n = D_{2^k} + R_k D_m,
\quad
n = 2^k + m,
\quad
m \in 1 : 2^k,
\end{equation}
where $R_k \equiv W_{2^k}$ are the {\it Rademacher functions};

\begin{equation}
\label{Eq:WFS-3}
D_{2^k} (g)
=
\begin{cases}
2^k,      & \text{$g \in \Delta_0^{(k)}$},
\\
0       & \text{otherwise}.
\end{cases}
\end{equation}
\begin{equation}
\label{Eq:WFS-5}
D_N
=
\sum\limits_{j=1}^s
R_{k_1} \ldots R_{k_{j-1}}
D_{2^{k_j}}, \quad N = 2^{k_1} + \ldots + 2^{k_s}, \
k_1 > \ldots > k_s,
\end{equation}
\begin{equation}
\label{Eq:Pl12-lemm5}
D_N(g) = 0, \quad N = 2^{k_1}+\ldots+2^{k_s}, \quad k_1 > \ldots > k_s, \quad g \in \mathbb{G}\setminus \mathbb{G}_{k_s}.
\end{equation}

The $\mathbf{N}$-th \emph{Dirichlet kernel} for the $d$-dimensional Walsh system
is
\begin{equation}
\label{Eq:WFS-11}
D_{\mathbf{N}} (\mathbf{g})
:=
\sum_{ \mathbf{n} < \mathbf{N} }
W_{\mathbf{n}} (\mathbf{g})
=
\prod\limits_{l=1}^d
D_{N^l} (g^l),
\quad
\mathbf{N} \in \mathbb{N}^d.
\end{equation}

\subsubsection{}
\label{kvazi-mese}
{\it Quasimeasures} on the group $\mathbb{G}^d$ are
finitely additive set functions
$\tau \colon \mathcal{B} \to \mathbb{C}$,
where $\mathcal{B}$ is a semiring
consisting of $\emptyset$ and all dyadic cubes.
Any quasimeasure $\tau$ can be extended
to the ring generated by $\mathcal{B}$;
if $\tau$ is nonnegative,
it can be extended to a $\sigma$-additive measure
on the $\sigma$-algebra of Borel subsets of $\mathbb{G}^d$.

It is easy to verify that a set function
$\tau \colon \mathcal{B} \to \mathbb{C}$
is a quasimeasure if and only if
\begin{equation}
\label{Eq:QM-Char-Prop}
\tau
\big( \Delta^{(k)}_{\mathbf{m}} \big)
=
\sum\limits_{ \boldsymbol{\sigma} \in \{ 0, 1 \}^d }
\tau
\big( \Delta^{(k+1)}_{ 2 \mathbf{m} + \boldsymbol{\sigma} } \big)
\;\;
\text{for all admissible $k$ and $\mathbf{m}$}.
\end{equation}

The set of all quasimeasures is isomorphic as a linear space
to the set of all series \eqref{series};
the canonical isomorphism
is established by the mapping that sends the series \eqref{series}
to the quasimeasure $\tau$,
\begin{equation}
\label{Eq:Canon-Iso}
\begin{split}
\tau \big( \Delta^{(k)} \big)
:=
&
\sum\limits_{\mathbf{n} \le 2^k \mathbf{1}}
\int\limits_{ \Delta^{(k)} }
c_{\mathbf{n}} W_{\mathbf{n}} ( \mathbf{g} ) \, d \mu ( \mathbf{g} )
\\
=
&
2^{-kd}  S_{2^k} ( \Delta^{(k)} ),
\end{split}
\end{equation}
which we say is generated by the given series.

The {\it support} of a quasimeasure $\tau$
is the (closed) set $F = \mathbb{G}^d \setminus G$,
where $G$ is the union of all possible dyadic cubes
$\Delta_0$ such that $\tau (\Delta)=0$
for all dyadic cubes $\Delta \subset \Delta_0$.
Notation: $\mathrm{supp} \, \tau$.

For more details, see \cite[2.3]{MP-EMJ-2019}; \cite{Schipp-and-Co};
\cite[Ch.~4]{VA-Hung-2004}.


\subsubsection{}
\label{Subsub:F-Tau-05}

If a set $F \subset \mathbb{G}^d$ is closed
and the series \eqref{series} converges over cubes to zero on
$\mathbb{G}^d \setminus F$, then $\mathrm{supp} \, \tau \subset F$
for the quasimeasure $\tau$ generated by this series.
For a proof, see, e.g., in~\cite[Lemma~1]{plotnikov-07-2}.

\subsubsection{}
We define a multi-valued ``contraction'' operator $\mathrm{C}_q$ on
the dyadic group $\mathbb{G}$, namely
$$ g = \bigoplus_{k=0}^{\infty} g_k e_k \mapsto \mathrm{C}_q(g) := \left\{ \widetilde{g}\in \mathbb{G}\colon \widetilde{g} = \bigoplus_{k=0}^{q-1}\widetilde{g}_k e_{k}
+ \bigoplus_{k=0}^{\infty} g_k e_{k+q} \right\}. $$
The multi-valued nature of the operator is explained by the non-injectivity of the dilation operator \cite{Luk09} on compact zero-dimensional groups, which, in turn, is due to their specific ``periodic'' structure. Note that for locally compact zero-dimensional groups, the dilation operator is injective, and therefore the ``contraction'' is uniquely defined.

Let $q\le p$, then the equality $\mathrm{C}_q(g) = \mathrm{C}_p(x)$ means that
\begin{equation} \label{r_2}
g \in  \mathrm{C}_{p-q}(x).
\end{equation}

\subsection{Auxiliary statements}
We denote $\mathbf{e}_k^{\boldsymbol{\sigma}} = (\sigma_1e_k, \ldots, \sigma_de_k)$, where $e_k$ are the generators of the group $\mathbb{G}$, and $\boldsymbol{\sigma} \in \Sigma^d$. Lemma \ref{lem-1} generalizes Lemma 3 from \cite{plotnikov-2010}, where it was proved for the case $M_1 = 2^k$.
\begin{lemma}
\label{lem-1}
Let $M_1, M_2 \in \mathbb{N}$, such that
\[
M_1 = 2^{k_1}+\ldots+2^{k_l}, \quad M_2 = M_1 + r, \quad k_1>k_2 \ldots > k_l, \quad   r < 2^{k_l }.
\]
Then for the cubic partial sums $S_N$ of series (1.1) and for all $\mathbf{g} \in \mathbb{G}^d$, the following equality holds
\[
\sum_{\boldsymbol{\sigma}^1, \ldots,  \boldsymbol{\sigma}^l \in \Sigma^d_2}  \left[ S_{M_2+1}\left( \mathbf{g} \oplus \left( 
\bigoplus_{j=1}^{l}
\mathbf{e}^{\boldsymbol{\sigma}^j}_{k_j+1} \right) \right) - S_{M_1}\left( \mathbf{g} \oplus \left( 
\bigoplus_{j=1}^{l}
\mathbf{e}^{\boldsymbol{\sigma}^j}_{k_j+1} \right) \right)  \right]
\]
\begin{equation} \label{eq-0-0}
    = 2^{l(d-1)} \sum_{\mathbf{n}=M_1 \cdot \mathbf{1}}^{M_2 \cdot \mathbf{1}} c_{\mathbf{n}} W_{\mathbf{n}}(\mathbf{g}).
\end{equation}
\end{lemma}
\begin{proof}
Let us denote, \[W_i= \{{\bf n} \colon \ 2^{k_1}+\ldots+2^{k_i} \le n_q \le M_2 \quad \text{at least for one $q$}\}\] and \[V_i=\{{\bf n} \colon \ (2^{k_1}+\ldots+2^{k_i}) \cdot \mathbf{1} \le {\bf n} \le M_2 \cdot \mathbf{1}\}.\]
    Let us prove by induction on $i$ the following equality:
   \begin{equation} \label{eq-0-2}
      \sum_{\boldsymbol{\sigma}^1, \ldots,  \boldsymbol{\sigma}^i \in \Sigma^d_2}  W_{{\bf n}}\left( \mathbf{g} \oplus
\bigoplus_{j=1}^{i}
\mathbf{e}^{\boldsymbol{\sigma}^j}_{k_j+1} \right) = \begin{cases}
2^{(d-1)i} W_{{\bf n}}\left( \mathbf{g}  \right), & \text{if} \  {\bf n} \in V_i, \\
0, & \text{if}  \  {\bf n} \in W_i \setminus V_i.
\end{cases}
   \end{equation}
From equality \eqref{eq-0-2} follows \eqref{eq-0-0}, since $V_l = \{\mathbf{n}\colon M_1\cdot \mathbf{1} \le \mathbf{n} \le M_2\mathbf{1}\}$, and
\[
\left[ S_{M_2+1}\left( \mathbf{g} \oplus \left( 
\bigoplus_{j=1}^{l}
\mathbf{e}^{\boldsymbol{\sigma}^j}_{k_j+1} \right) \right) - S_{M_1}\left( \mathbf{g} \oplus \left( 
\bigoplus_{j=1}^{l}
\mathbf{e}^{\boldsymbol{\sigma}^j}_{k_j+1} \right) \right)  \right] = \sum_{\mathbf{n} \in W_l} W_{\mathbf{n}}({\mathbf{g}}).
\]

The base of induction (the case $l = 1$ in \eqref{eq-0-2}) is contained in Lemma 3 from \cite{plotnikov-2010}.
By the induction hypothesis, we have
   \[ \sum_{\boldsymbol{\sigma}^1 \in \Sigma^d_2} \ldots \sum_{\boldsymbol{\sigma}^{l-1} \in \Sigma^d_2} W_{{\bf n}}\left( (\mathbf{g} \oplus \mathbf{e}^{\boldsymbol{\sigma}^l}_{k_l+1} ) \oplus  \bigoplus_{j=1}^{l-1}
\mathbf{e}^{\boldsymbol{\sigma}^j}_{k_j+1}  \right) = \]
\begin{equation} \label{eq-0-01}
   = \begin{cases}
2^{(d-1)(l-1)} W_{{\bf n}}\left( \mathbf{g} \oplus \mathbf{e}^{\boldsymbol{\sigma}^l}_{{k_l+1}} \right), & \text{if} \  {\bf n} \in V_{l-1}, \\
0, & \text{if}  \  {\bf n} \in W_{l-1} \setminus V_{l-1}.
\end{cases}
\end{equation}
For  ${\bf n} \in V_{l-1}$:
\begin{equation} \label{eq-0-12}
    \mathbf{n} = 2^{k_1}\mathbf{1} + \ldots 2^{k_{l-1}}\mathbf{1} + \mathbf{s}, \qquad  \mathbf{s} \le 2^{k_{l}}\mathbf{1} + r\mathbf{1}.
\end{equation}
Then
\begin{equation} \label{eq-0-1}
    \sum_{\boldsymbol{\sigma}^l \in \Sigma^d_2} W_{{\bf n}}\left( \mathbf{g} \oplus \mathbf{e}^{\boldsymbol{\sigma}^l}_{k_l+1}   \right) =  R_{k_1 {\bf 1}}\ldots R_{k_{l-1} {\bf 1}}(\mathbf{g}) \sum_{\boldsymbol{\sigma}^l \in \Sigma_2^d} W_{\bf s} \left( \mathbf{g} \oplus \mathbf{e}^{\boldsymbol{\sigma}^l}_{k_l+1}   \right).   
\end{equation}
By Lemma 3 from \cite{plotnikov-2010}, expression \eqref{eq-0-1} equals
\begin{equation} \label{eq-0-111}
    \begin{cases}
2^{d-1} W_{\bf n}(\mathbf{g}), & \text{if} \ 2^{k_l} \mathbf{1}  \le \mathbf{s} \le 2^{k_l}  \mathbf{1} +r  \mathbf{1}  \\
0, & \text{otherwise}.
\end{cases}
\end{equation}
Substituting $\mathbf{s}$ from the first equality \eqref{eq-0-111} into \eqref{eq-0-12}, we have $\mathbf{n} \in V_l$ for such $\mathbf{s}$. Then for $\mathbf{n} \in V_l$
\[
\sum_{\boldsymbol{\sigma}^1 \in \Sigma^d_2} \ldots \sum_{\boldsymbol{\sigma}^l \in \Sigma^d_2} W_{{\bf n}}\left( \mathbf{g} \oplus
\bigoplus_{j=1}^{l}
\mathbf{e}^{\boldsymbol{\sigma}^j}_{k_j+1} \right)=\] 
\[= 2^{(d-1)(l-1)}\sum_{\boldsymbol{\sigma}^l \in \Sigma^d_2}  W_{{\bf n}}\left( \mathbf{g} \oplus \mathbf{e}^{\boldsymbol{\sigma}^l}_{{k_l+1}} \right) = 2^{(d-1)l} W_{{\bf n}}\left( \mathbf{g} \right).
\]
Thus, for $\mathbf{n} \in V_l$ equality \eqref{eq-0-2} is proved. In the case $\mathbf{n} \in W_l\setminus V_l$ we obtain either $\mathbf{n} \in W_{l-1}\setminus V_{l-1}$, and then \eqref{eq-0-2} follows from \eqref{eq-0-01}, or $\mathbf{n} \in V_{l-1} \setminus V_l$, and then \eqref{eq-0-2} follows from \eqref{eq-0-111}.

Equality \eqref{eq-0-2} is proved, and with it the lemma.
\end{proof}

Lemma \ref{lem-2} generalizes Lemma 4 from \cite{plotnikov-2010}. The proofs of these lemmas are similar.
\begin{lemma}
\label{lem-2}

Let $\Delta$ be a dyadic cube of rank $s$, $E \subseteq \Delta$ be a Borel set, with $\mu(E) > 0$. Let some $l\in \mathbb{N}$ be fixed. Then there exists a natural number $k_0 \geq s$ such that for all $k_1 > \ldots>k_l \geq k_0$ there exists a point $\mathbf{g}_{k_1, \ldots, k_l} \in \Delta$ for which the following condition holds

\begin{equation} \label{eq-lem-2}
    \mathbf{g}_{k_1, \ldots, k_l} \oplus \left(\bigoplus_{i=1}^{l} \mathbf{e}^{\boldsymbol{\sigma}^i}_{k_i} \right) \in E  \quad \text{for all } \boldsymbol{\sigma}^i \in \Sigma^d, \text{where } i \in \{1, \ldots, l\}.
\end{equation}
\end{lemma}
\begin{proof}
Since $\mu(E)>0$, the set $E$ has density points. Therefore, there exists a cube $\Delta_0$ of rank $k_0$ such that
\[
\frac{\mu(\Delta_0 \cap E)}{\mu(\Delta_0)} > 1- \frac{1}{2^{ld}}.
\]
Fix arbitrary $k_1 > \ldots > k_l \ge k_0$. For all combinations $\boldsymbol{\sigma}^i \in \Sigma^d$, where $i \in \{1, \ldots, l\}$, consider the sets (there are $2^{ld}$ of them in total): \[
E_0^{\boldsymbol{\sigma}^1, \ldots, \boldsymbol{\sigma}^l} := (\Delta_0 \cap E) \oplus \left(\bigoplus_{i=1}^{l} \mathbf{e}^{\boldsymbol{\sigma}^i}_{k_i} \right).
\]  
Then $E_0^{\boldsymbol{\sigma}^1, \ldots, \boldsymbol{\sigma}^l} \subset \Delta_0$ for all $\boldsymbol{\sigma}^i \in \Sigma^d$. Let us denote
\[
L := \bigcap_{\boldsymbol{\sigma}^1 \in \Sigma^d}\ldots\bigcap_{\boldsymbol{\sigma}^l \in \Sigma^d} E_0^{\boldsymbol{\sigma}^1, \ldots, \boldsymbol{\sigma}^l}.
\]
Using the invariance of the Haar measure under translation, we obtain
\[
\begin{split}
&\mu(L) = \mu(\Delta_0) - \mu(\Delta_0 \setminus L) \ge \mu(\Delta_0) - \sum_{\boldsymbol{\sigma}^1, \ldots, \boldsymbol{\sigma}^l \in \Sigma^d} \mu(\Delta_0 \setminus E_0^{\boldsymbol{\sigma}^1, \ldots, \boldsymbol{\sigma}^l}) > \\
& > \mu(\Delta_0) - \sum_{\boldsymbol{\sigma}^1, \ldots, \boldsymbol{\sigma}^l \in \Sigma^d} \frac{1}{2^{dl}}\mu(\Delta_0) = 0.
\end{split}
\]
Therefore, $\mu(L) > 0$ and $L \neq \emptyset$. Choose an arbitrary point $\mathbf{g}_{k_1, \ldots, k_l} \in L$. Then $\mathbf{g}_{k_1, \ldots, k_l} \in E_0^{\boldsymbol{\sigma}^1, \ldots, \boldsymbol{\sigma}^l}$ for all possible combinations $\boldsymbol{\sigma}^i \in \Sigma^d$. Consequently, taking into account that the operation $\oplus$ on the dyadic group $\mathbb{G}$ coincides with $\ominus$, the point $\mathbf{g}_{k_1, \ldots, k_l}$ satisfies \eqref{eq-lem-2}. The lemma is proved.
\end{proof} 

\begin{lemma}
\label{lem-4}
Let $P = \Delta^{(k_1)} \times \ldots \times \Delta^{(k_d)} \subset \mathbb{G}^d$ be some dyadic parallelepiped, and $E = \mathrm{supp} \, \tau$ be the support of some quasimeasure $\tau$. Then if
\begin{equation} \label{eq-lem-4}
    W_{\bf N} ({\bf g}) = 1 \qquad \text{for ${\bf g} \in E \cap P$},
\end{equation}
then \[
\int\limits_{P} W_{\bf N} ({\bf g}) d\tau = \tau(P).
\]
\end{lemma}
\begin{proof}
Let ${\bf N} = (N^1,\ldots, N^d)$ и $k = \max\limits _{i \in 1 : d} \lfloor \log_2 N^i\rfloor$.
    \[ 
\begin{split} 
& \int\limits_{P} 
W_{ \mathbf{N} }({\bf g}) d \tau 
= 
\sum_{ \Delta^{(k+1)} \subset P } 
\int\limits_{ \Delta^{(k+1)} } 
W_{ \mathbf{N} }({\bf g}) d \tau 
= 
\sum_{ \Delta^{(k+1)} \subset P  } 
W_{\mathbf{N}} (\Delta^{(k+1)}) \tau ( \Delta^{(k+1)} )
\\ 
& 
\stackrel{\star}{=}
\sum_{ \Delta^{(k+1)} \subset P \; \wedge \; \Delta^{(k+1)} \cap E \neq \emptyset } 
W_{\mathbf{N}} (\Delta^{(k+1)}) \tau ( \Delta^{(k+1)} )
\\ 
& 
 \stackrel{\eqref{eq-lem-4}}{=} \sum_{ \Delta^{(k+1)} \subset P \; \wedge \; \Delta^{(k+1)} \cap E \neq \emptyset } 
 \tau ( \Delta^{(k+1)} ) 
= 
 \tau (P \cap E) 
\stackrel{\star}{=}
 \tau(P), 
\end{split}
\]
where in the transitions marked by $\star$, we use that $E$ is the support of the quasimeasure $\tau$.
The lemma is proved.
\end{proof}

Fix an arbitrary partition $\{1, \ldots d\} = \{i_1, \ldots, i_m\} \sqcup \{j_1, \ldots, j_{d-m}\}$, where $m \le d-1$. Then throughout the paper, each element ${\bf g} = (g^1, \ldots, g^d)$ of the group $\mathbb{G}^d$ will be written as ${\bf g} = ({\bf g}^{*}, {\bf g}_{*})$, where
\begin{equation} \label{g*}
  {\bf g}^{*} = (g^{j_1}, \ldots, g^{j_{d-m}}), \quad {\bf g}_{*} = (g^{i_1}, \ldots, g^{i_m}).
\end{equation}
Similarly, for dyadic cubes and indices
\begin{equation} \label{Delta*}
\Delta = \Delta^{*}\times\Delta_{*}, \qquad \mathbf{n} = (\mathbf{n}^{*}, \mathbf{n}_{*}),
\end{equation}
where $*$ above corresponds to the coordinates $\{j_1, \ldots, j_{d-m}\}$, and $*$ below corresponds to $\{i_1, \ldots, i_m\}$.
\begin{lemma}
\label{lem-5} Let $\tau(\Delta) = C$, then there exists a point $\boldsymbol{\xi} \in \mathbb{G}^{m}$ such that
$$ \tau(\Delta^* \times \Delta^{(k)}(\boldsymbol{\xi})) > \frac{1}{2^{m(k-k_0)}} C.$$
\end{lemma}
\begin{proof}
    One can construct a nested sequence of parallelepipeds that contracts
to a $(d-m)$-dimensional cube. Partition $\Delta$ into $2^m$ parallelepipeds
$$ \Delta = \bigsqcup_{\Delta_*^{(k_0+1)} \subset \Delta_*} \Delta^* \times \Delta_*^{(k_0+1)}. $$
Since $\tau(\Delta)=C$, there exists a cube $\widetilde{\Delta}_{*}^{(k_0+1)}$ such that
$$ \tau(\Delta_* \times \widetilde{\Delta}_{*}^{(k_0+1)}) > \frac{C}{2^{m}}. $$
Repeating this argument multiple times, we obtain:
$$ \dots \subset \Delta^* \times \widetilde{\Delta}_{*}^{(k_0+k)} \subset \dots \subset \Delta^* \times \widetilde{\Delta}_{*}^{(k_0+1)} \subset \Delta, \quad \text{где $\tau(\Delta^* \times \widetilde{\Delta}_{*}^{(k_0+k)}) > \frac{1}{2^{mk}} C$}.$$
Then
$$ \Delta^* \times \boldsymbol{\xi} := \bigcap_{k=1}^{\infty} \Delta^* \times \widetilde{\Delta}_{*}^{(k_0+k)}. $$ 
The lemma is proved.
\end{proof}

\section{Main results} 
\label{S:Main-results}
\subsection{Dyadic planes on the group $\mathbb{G}^d$} \label{planes}
We distinguish classes of subgroups in the group $\mathbb{G}^{d}$.

Fix $d-m$ natural numbers $(q_j \in \mathbb{N})_{j=1}^{d-m} =:\mathbf{q}$.
Then
\begin{equation} \label{R_m}
    Q_{m, \mathbf{q}} = \left\{\mathbf{g} \in \mathbb{G}^d: \mathrm{C}_{q_1}(g^{j_1}) = \mathrm{C}_{q_2}(g^{j_2}) = \dots = \mathrm{C}_{q_{d-m}}(g^{j_{d-m}}), \mathbf{g}_* \in \mathbb{G}^m\right\}, \; m\le d-2
\end{equation}
is a certain subgroup isomorphic to $\mathbb{G}^{m+1}$. In the case where all shifts are trivial, $Q_{m, \mathbf{q}}$ is a "diagonal" subgroup
\begin{equation} \label{D_m}
    D_m = \{\mathbf{g} \in \mathbb{G}^d \colon g^{j_1}= \ldots=  g^{j_{d-m}},  \ {\bf g}_{*} \in \mathbb{G}^{m}\}, \quad m\le d-2.
\end{equation}
And we define the subgroup \(P_m\), which is the natural embedding of the group \(\mathbb{G}^m\) into the group \(\mathbb{G}^d\):
\begin{equation} \label{P_m}
P_m = \{{\bf g}\in \mathbb{G}^d \colon{\bf g}^{*}= \mathbf{0}, \ {\bf g}_{*} \in \mathbb{G}^{m}\}, \quad m\le d-1.
\end{equation}
Fix some element $\mathbf{x} \in \mathbb{G}^{d}$ and consider the cosets of the subgroups \eqref{R_m}, \eqref{D_m}, \eqref{P_m}:
\begin{equation} \label{all}
D_m^{\mathbf{x}} := D_m \oplus  \mathbf{x}, \quad Q_{m, \mathbf{q}}^{\mathbf{x}} := Q_{m, \mathbf{q}} \oplus  \mathbf{x}, \quad P_m^{\mathbf{x}} := P_m \oplus  \mathbf{x}.
\end{equation}

The following remarks provide a geometric interpretation of the sets $D^{\mathbf{x}}_m$, $Q^{\mathbf{x}}_{m, \mathbf{q}}$, and $P^{\mathbf{x}}_m$.
\begin{remark} Let $\mathbf{x}^*$ be defined according to \eqref{g*}.

If $\mathbf{x}^{*}=\textbf{0}$, then $D_m$ is an $(m+1)$-dimensional diagonal plane (see Fig. \ref{fig:3}a).

If $\mathbf{x}^{*} = \textbf{1}$, then $D_m^{\mathbf{x}}$ is an "anti-diagonal" plane (see Fig. \ref{fig:3}b).

If $\mathbf{x}^{*}$ is the "right" preimage of $\mathbf{t}_+$ under the mapping \eqref{otobr-F} for some dyadic rational point $\mathbf{t}$, then $D_m^{\mathbf{x}}$ is a set of planes parallel to $D_m$ and shifted from the diagonal plane by the element $\mathbf{x}$ (see Fig. \ref{fig:3}c).

If $\mathbf{x}^{*}$ is the "left" preimage of $\mathbf{t}_-$, then $D_m^{\mathbf{x}}$ is a set of planes parallel to the anti-diagonal plane.

In the case of a dyadic irrational point $\mathbf{x}^{*}$, the geometric interpretation of the set $D_m^{\mathbf{x}}$ becomes more complicated (see Fig. \ref{fig:3}d).
\end{remark}

\begin{remark} Let $\mathbf{x}^*$ be defined according to \eqref{g*}.

If $\mathbf{x}^* = \mathbf{0}$, then the set $Q_{m, \mathbf{q}}$ is a union of $2^{q_1 + \ldots + q_{d-m}}$ $(m+1)$-dimensional planes. These planes are parallel and located obliquely in $[0,1]^d$ (see Fig. \ref{fig:4}a).

If $\mathbf{x}^{*}$ is the "right" preimage of $\mathbf{t}_+$ under the mapping \eqref{otobr-F} for some dyadic rational point $\mathbf{t}$, then $Q_{m, \mathbf{q}}^{\mathbf{x}}$ is a set of planes parallel to $Q_{m, \mathbf{q}}$ (see Fig. \ref{fig:4}b).

If $\mathbf{x}^{*}$ is the "left" preimage of $\mathbf{t}_-$, then $Q_{m, \mathbf{q}}^{\mathbf{x}}$ is a set of planes parallel to $Q_{m, \mathbf{q}}^{\mathbf{1}}$ (see Fig. \ref{fig:4}c).

In the case of a dyadic irrational point $\mathbf{x}^{*}$, the geometric interpretation of the set $Q_{m, \mathbf{q}}^{\mathbf{x}}$ becomes more complicated (see Fig. \ref{fig:4}d).
\end{remark}

\begin{remark} $P^{\mathbf{x}}_m = \{\mathbf{x}\} \times \mathbb{G}^m$ is an $m$-dimensional plane parallel to the coordinate planes (see Fig. \ref{fig:5}).
\end{remark}

We will call the sets defined by formula \eqref{all} {\it dyadic planes} on the group $\mathbb{G}^d$.

\subsection{On the connection between quasimeasures on the group $\mathbb{G}^d$ and Walsh series}
\label{p-2-2}
The following two theorems establish a connection between the convergence of special cubic sums of a series and certain properties (\eqref{Eq:1:2} and \eqref{Eq:2:2}) of the generated quasimeasure. Conditions \eqref{Eq:1:2} and \eqref{Eq:2:2} can naturally be interpreted as some kind of continuity of the quasimeasure generated by the series. The choice of the type of continuity depends on the type of convergence of the series under consideration. Thus, under the assumption of rectangular convergence of the series, the generated quasimeasure is continuous in the sense of Sacks (i.e., $\tau(\Delta^{(k)}) \rightarrow 0$ as $\mu_d(\Delta^{(k)}) \rightarrow 0$).
In the case of $\lambda$-convergence and even more so in the case of convergence over cubes, the convergence of the series does not guarantee such continuity, and there is a need for more subtle types of continuity. Such continuities were introduced by M.G. Plotnikov (see, e.g., \cite{plotnikov-07-1}, \cite{plotnikov-2010}, \cite{MP-EMJ-2019}). In \cite{plotnikov-07-1}, local "chessboard"$\ $ continuity was considered, which, as proved in \cite{plotnikov-skvortsov}, coincides with classical continuity in the one-dimensional case, and in \cite{plotnikov-2010}, nonlocal "chessboard"$\ $ continuity was considered. Other types of continuity were considered in \cite{MP-EMJ-2019} and essentially in \cite{kazakova-plotnikov-2025-3}.

For more details on the continuity of quasimeasures, see the survey paper \cite{plotnikov-skvortsov}.

Condition \eqref{Eq:1:2} defines an intermediate continuity between nonlocal
"chessboard"$\ $ continuity and classical local continuity at a point. Essentially, continuity is considered not at a point, but in some lower-dimensional plane parallel to the coordinate planes, and on this plane the quasimeasure is taken with signs according to the "chessboard"$\ $order.

Property \eqref{Eq:2:2} of a quasimeasure corresponds to a type of nonlocal continuity that is structurally different from the "chessboard"$\ $one.

Note that the nonlocal types of continuity of a quasimeasure are essentially the Walsh--Fourier coefficients of this quasimeasure restricted to some dyadic cube (see \cite{kazakova-plotnikov-2025-3}). Similar quantities are sometimes called local Fourier coefficients (see, e.g., \cite{golubov-efimov-skvortsov}, formula (9.2.4)).

Theorems \ref{th-2-1} and \ref{th-2-2} generalize Theorem 2 from \cite{plotnikov-2010} in different directions.
\begin{theorem}
\label{th-2-1}
Let an arbitrary set $\{i_1, \ldots, i_m\} \subset \{1, \ldots, d\}$ of cardinality $ m \le d-1$,
a Borel set $A \subset \mathbb{G}^d$ of positive measure, and a series $(S)$ of the form \eqref{series} be given.
Assume that for all natural numbers $s$, for the partial sums of this series at each point $\mathbf{g} \in A$, there exist (finite) limits    
\begin{equation} 
\label{Eq:1:1} 
\lim_{k \rightarrow \infty} S_{ 2^{k} {\bf 1} + 2^s {\bf 1}}  ( \mathbf{g} ) = \lim_{k \rightarrow \infty} S_{ 2^{k} {\bf 1}}  ( \mathbf{g} ) 
\end{equation}
Then for the quasimeasure $\tau$ generated by the series $(S)$, the following holds 
\begin{equation} 
\label{Eq:1:2}
\lim\limits_{k \to \infty} 2^{mk}
\int\limits_{\Delta \times \Delta^{(k)}(\boldsymbol{\xi})} R_{k {\bf 1}} (\mathbf{g}^{*})d \tau = 0,
\end{equation}
where $\mathbf{g}^{*}$ is defined in \eqref{g*}, and $\boldsymbol{\xi} \in \mathbb{G}^m$ and  $\Delta \in \mathbb{G}^{d-m}$ are such that for every partition of the set ${i_1, \ldots, i_m} = J \sqcup \overline{J}$, the following condition holds
 \begin{equation} \label{Eq:1:11}
\mu_{d+|J|-m}\left(\left[\Delta \times \Delta(\boldsymbol{\xi}^J) \times \boldsymbol{\xi}^{\overline{J}}\right] \cap A\right) > 0.
\end{equation}
Here $\boldsymbol{\xi}^{J} \in \mathbb{G}^{|J|}$ and $\boldsymbol{\xi}^{\overline{J}} \in \mathbb{G}^{|\overline{J}|}$ correspond to the coordinates of the vector $\boldsymbol{\xi}$ belonging to the sets $J$ and $\overline{J}$, respectively; $\Delta(\boldsymbol{\xi}^J)$ is a $|J|$-dimensional cube containing the point $\boldsymbol{\xi}^J$, of the same rank as $\Delta$. The set $\Delta \times \Delta(\boldsymbol{\xi}^J) \times \boldsymbol{\xi}^{\overline{J}}$ is a subset of some $(d+|J|-m)$-dimensional plane parallel to the coordinate planes.
\end{theorem}
\begin{proof}
We carry out the proof by induction on $m$. The base of induction (the case $m=0$) is Theorem 2 from the paper \cite{plotnikov-2010}. Suppose the statement holds for all $0 \le l < m$. We show that then it also holds for $m$. For a fixed $\boldsymbol{\xi} \in \mathbb{G}^m$, denote
\[
S^{*}_{N}(\mathbf{g}^{*}) = \sum_{\mathbf{n}^{*}<N\cdot {\bf 1}} \left[\sum_{\mathbf{n}_{*}<N\cdot {\bf 1}} c_{\bf n} W_{\mathbf{n}_{*}} (\boldsymbol{\xi})\right] W_{\mathbf{n}^{*}}(\mathbf{g}^{*}).
\]
Consider the section of the set $A$
\[
A^{*} := A \cap \left[\mathbb{G}^{d-m}\times \boldsymbol{\xi}\right].
\]
The set $A^{*}$ is also a Borel set, and the following holds
\begin{equation} \label{Eq:1:3}
    \lim_{k \rightarrow \infty} S^{*}_{2^{k} +2^s} (\mathbf{g}^{*}) =  \lim_{k \rightarrow \infty} S^{*}_{2^{k}} (\mathbf{g}^{*}), \quad \text{for all \ } \mathbf{g}^{*} \in A^{*} \text{\ and \ } s\in \mathbb{N}.
\end{equation}
Let $s$ be the rank of the cube $\Delta$ from the condition of the theorem. By Egorov's theorem, there exists a set $E^{*} \subset A^{*}$ such that the convergence in \eqref{Eq:1:3} is uniform. Then, by Lemma 4 from \cite{plotnikov-2010}, there exists a point $\boldsymbol{\eta} \in \mathbb{G}^{d-m}$ such that
\[
\boldsymbol{\eta}  \oplus {\bf e}^{\boldsymbol{\sigma}}_{k+1} \in E^{*} \quad \text{for $\boldsymbol{\sigma} \in \Sigma^{d-m}_2$. }
\]
Denote
\[T_k(\boldsymbol{\eta}, \boldsymbol{\xi}) = \sum_{\boldsymbol{\sigma} \in \Sigma_2^{d-m}} \left[ S_{2^k + 2^s} (\boldsymbol{\eta} \oplus \mathbf{e}_{k+1}^{\boldsymbol{\sigma}}, \boldsymbol{\xi}) - S_{2^k} (\boldsymbol{\eta} \oplus \mathbf{e}_{k+1}^{\boldsymbol{\sigma}}, \boldsymbol{\xi}) \right].
\]
By Lemma 3 of \cite{plotnikov-2010}
\[
T_{k}(\boldsymbol{\eta}, \boldsymbol{\xi}) = 2^{d-m} \sum_{2^{k} \leq \mathbf{n}^* < 2^{k} + 2^s} \left[ \sum_{\mathbf{n}_* < 2^{k} + 2^s} c_{\mathbf{n}} W_{\mathbf{n}_*} (\boldsymbol{\xi}) \right] W_{\mathbf{n}^*}(\boldsymbol{\eta}).
\]
Considering the choice of the point $\boldsymbol{\eta}$ and the uniform convergence on the set $E^{*}$, we obtain
\begin{equation} \label{Eq:1:4}
    T_{k}(\boldsymbol{\eta}, \boldsymbol{\xi}) \rightarrow 0.
\end{equation}

Using formulas \eqref{Eq:WFS-2} and \eqref{Eq:WFS-3} for the Dirichlet kernel, we obtain that
\begin{align*}
&T_k(\boldsymbol{\eta}, \boldsymbol{\xi}) = \\
&\int\limits_{\mathbb{G}^d} \left( \prod_{\ell=1}^{d-m} [D_{2^k + 2^s} ({\eta}_{j_{\ell}} \oplus {g}_{j_{\ell}}) - D_{2^k} ({\eta}_{j_{\ell}} \oplus {g}_{j_{\ell}})] \right) \left( \prod_{\ell=1}^{m} D_{2^k + 2^s} ({\xi}_{i_{\ell}} \oplus {g}_{i_{\ell}}) \right) d{\tau}(\mathbf{g})= \\
&R_{k {\bf 1}}(\boldsymbol{\eta}) \int\limits_{\mathbb{G}^d} 2^{(d-m)s}  R_{k{\bf 1}}(\mathbf{g}^*)I(\mathbf{g}^* \in \Delta)\times \\
&\times \sum_{q=0}^m 2^{(m-q)k + q \cdot s}  \sum_{J: |J|=q} R_{k {\bf 1}}(\boldsymbol{\xi}^{J}) R_{k {\bf 1}}(\mathbf{g}^J_*) I(\mathbf{g}_*\in \Delta^{(s)}(\boldsymbol{\xi}^J) \times \Delta^{(k)}(\boldsymbol{\xi}^{\overline{J}}))  d{\tau} = \\
\end{align*}
\vspace{-12mm}
\begin{equation} \label{T_k}
    R_{k {\bf 1}}(\boldsymbol{\eta})  \sum_{q=0}^m \sum_{J: |J| =q} 2^{s(d-(m-q))}  R_{k {\bf 1}}(\boldsymbol{\xi}^{J}) Q^J_k,
\end{equation}
where \begin{equation*} 
    Q^J_k = 2^{(m-q)k}\int\limits_{\mathbb{G}^d} R_{k {\bf 1}}(\mathbf{g}^*, \mathbf{g}^{J}_*) I(\mathbf{g}^* \in \Delta)I(\mathbf{g}_*\in \Delta^{(s)}(\boldsymbol{\xi}^J) \times \Delta^{(k)}(\boldsymbol{\xi}^{\overline{J}})) d{\tau} 
\end{equation*} 
and $\sum_{J: |J| =q}$ is taken over all subsets $J \subset \{i_1, \ldots, i_m\}$ of cardinality $q$.

Since for all sets $J$ of positive cardinality the cardinality of the complement $\overline{J}$ is less than $m$, then by the induction hypothesis, the following holds

\begin{equation}
\label{Eq:1:5}
    Q^J_k =  2^{(m-q)k} \int\limits_{\Delta \times \Delta(\boldsymbol{\xi}^{J})\times\Delta^{(k)}(\boldsymbol{\xi}^{\overline{J}})} R_{k {\bf 1}}(\mathbf{g}^*, \mathbf{g}^{J}_*) d{\tau}  \rightarrow 0 \quad \text{as $k\rightarrow\infty$}.
\end{equation}

According to \eqref{T_k}, \(Q^{\emptyset}_k\) is a finite linear combination of the quantity \(T_k\) and the quantities \(Q^{J}_k\) for all nonempty \(J\), and the coefficients of this linear combination are bounded. Taking into account \eqref{Eq:1:4} and \eqref{Eq:1:5}, we obtain
\[
    Q^{\emptyset}_k =  2^{mk} \int\limits_{\Delta \times \Delta^{(k)}(\boldsymbol{\xi})} R_{k {\bf 1}}(\mathbf{g}^*) d{\tau}  \rightarrow 0 \quad \text{as $k\rightarrow\infty$}.
\]
Thus, formula \eqref{Eq:1:2} is proved. This completes the proof.
\end{proof}
\begin{remark}
    It can be seen from the proof that in Theorem \ref{th-2-1} the sequence of indices
\(\{k\}_{k=1}^{\infty}\) can be replaced by any of its subsequences
\(\{k_i\}_{i=1}^{\infty}\).
\end{remark}
We give an example of a set $A$ that satisfies condition \eqref{Eq:1:11} for any cube $\Delta \in \mathbb{G}^{d-m}$ and any point $\boldsymbol{\xi} \in \mathbb{G}^m$.
\begin{exmpl} \label{exmpl} Let the set $\mathbb{G}^d \setminus A$ be a dyadic plane defined in \eqref{all}.

In this case, for every partition $\{i_1, \ldots, i_m\} = J \sqcup \overline{J}$ and every point $\boldsymbol{\xi} \in \mathbb{G}^m$, the set $P := \left[\mathbb{G}^d \setminus A\right] \cap \left[ \mathbb{G}^{d+|J|-m} \times \boldsymbol{\xi}^{\overline{J}} \right]$ is the intersection of a dyadic $m$-dimensional and a dyadic $(d+|J|-m)$-dimensional plane, and these planes do not coincide and are not embedded in one another, so their intersection is a dyadic plane of strictly lower dimension. Consequently,
\[
\mu_{d+|J|-m}\left(P \cap [\Delta\times \Delta(\boldsymbol{\xi})]\right) = 0 \quad \text{for all $\Delta \in \mathbb{G}^{d-m}$}.
\]
Passing to the complement of $P$ in the plane $\mathbb{G}^{d+|J|-m} \times \boldsymbol{\xi}^{\overline{J}}$, we obtain \eqref{Eq:1:11}.

\end{exmpl}

Theorem \ref{th-2-2} is an analogue of Theorem 4.4 from \cite{kazakova-plotnikov-2025-3}, which considered a stronger type of convergence, namely $\lambda$-convergence, and at the same time a sequence $\mathbf{N}_i$ with less restrictive conditions than in Theorem \ref{th-2-2}.

\begin{theorem} \label{th-2-2} 
Let a sequence $( N_i \in \mathbb{N}, \, i \in \mathbb{N})$,
a Borel set $A \subset \mathbb{G}^d$ of positive measure, and a series $(S)$ of the form \eqref{series} be given.
Suppose the following holds.  

$\circ$
The smallest index of a nonzero dyadic coefficient of the number $N_i$ tends to infinity as $i \to \infty$ (property $P$).

$\circ$
For all natural numbers $s$, for the partial sums of this series at each point $\mathbf{g} \in A$, there exist (finite) limits     
\begin{equation} 
\label{Eq:2:1} 
\lim_{i \rightarrow \infty} S_{ N_i  {\bf 1} + 2^s  {\bf 1}}  ( \mathbf{g} ) = \lim_{i \rightarrow \infty} S_{ N_i {\bf 1}}  ( \mathbf{g} ) 
\end{equation}

$\circ$ The number of nonzero coefficients in the expansions of $N_i$ is bounded:
\begin{equation} \label{Eq:2-0}
    \overline{\lim\limits_{i \rightarrow \infty}} \#N_i< \infty
\end{equation}

Then   
\begin{equation} 
\label{Eq:2:2}
\lim\limits_{i \to \infty} 
\int\limits_{\Delta} W_{N_i {\bf 1}} d \tau = 0 
\end{equation}
for any dyadic cube $\Delta$ such that $\mu (\Delta \cap A) > 0;$
$\tau$ is the quasimeasure generated by the original series.
\end{theorem}
\begin{proof}
Fix a dyadic cube $\Delta = \Delta^1 \times \ldots \times \Delta^d$ from the statement of the theorem
and let $s$ be its rank.
Set
\[ 
V_i 
:= 
\{ 
\mathbf{m} \in \mathbb{N}^d 
\colon 
N_i \mathbf{1} \le \mathbf{m} \le N_i \mathbf{1} + 2^s \mathbf{1} \}. 
\]
Since \eqref{Eq:2-0} holds, we can apply Lemmas \ref{lem-1} and \ref{lem-2}.
Take an arbitrary $\varepsilon > 0$. Then, using Lemmas \ref{lem-1}, \ref{lem-2} and repeating the beginning of the proof of Theorem 2 from \cite{plotnikov-2010}, we find a point $\mathbf{g} \in A$ such that
  
\begin{equation} 
\label{Eq:U-01}
\bigg| 
    \sum_{ \mathbf{n} \in V_i} 
    c_{ \mathbf{n} } W_{ \mathbf{n} }( \mathbf{g} ) 
\bigg|< \varepsilon   
\end{equation}
for all sufficiently large $i$.
Moreover,
\begin{equation} 
\label{Eq:U-02} 
\begin{split}
\sum_{ \mathbf{n} \in V_i} 
c_{ \mathbf{n} } W_{ \mathbf{n} }( \mathbf{g} ) 
& 
= 
\sum_{ \mathbf{n} \in V_i} 
W_{ \mathbf{n} }( \mathbf{g} )
\int\limits_{\mathbb{G}^d} 
W_{ \mathbf{n} } ( \mathbf{x} ) d \tau ( \mathbf{x} )  
= 
\int\limits_{\mathbb{G}^d} 
\sum_{ \mathbf{n} \in V_i}  
W_{\mathbf{n}} ( \mathbf{x} \oplus \mathbf{g} ) d \tau ( \mathbf{x} )
\\ 
& 
= 
\int\limits_{\mathbb{G}^d} 
\prod_{j=1}^d 
\big[ D_{ N_i + 2^q } ( x^j \oplus g^j ) - D_{ N_i } ( x^j \oplus g^j ) \big]
d \tau ( \mathbf{x} ). 
\end{split} 
\end{equation}

Since property $P$ holds, we can assume that the indices of all nonzero dyadic coefficients of the numbers $N_i$ are greater than $q$.
Using the last condition of the theorem and formulas \eqref{Eq:WFS-3} and \eqref{Eq:WFS-5}, we obtain
\begin{equation} 
\label{Eq:U-03}
\begin{split} 
D_{ N_i + 2^q }( x^j \oplus g^j ) - D_{N_i}( x^j \oplus g^j ) 
& 
= 
2^q W_{N_i} ( x^j \oplus g^j ) 
\cdot 
I ( x^j \oplus g^j \in \Delta^{(s)}_0 ) 
\\ 
& 
= 2^q W_{N_i} ( x^j \oplus g^j ) \cdot I( x^j \in \Delta^j).
\end{split} 
\end{equation} 

For large $i$, we obtain
\[
\begin{split} 
\left|
    2^{qd} \int\limits_{\Delta} W_{ \mathbf{N} } d \tau 
\right| 
& 
\stackrel{\eqref{Eq:U-03}}{=}
\left|
    \int\limits_{\mathbb{G}^d} 
    \prod_{i=1}^d 
    \big[ D_{ N_i + 2^q } ( x^j \oplus g^j ) - D_{ N_i } ( x^j \oplus g^j ) \big]
    d \tau ( \mathbf{x} ) 
\right| 
\\ 
& 
\stackrel{\eqref{Eq:U-02}}{=}
\left|
    \sum_{ \mathbf{n} \in V_i} 
    \tau_{ \mathbf{n} } W_{ \mathbf{n} }( \mathbf{g} )
\right| 
\stackrel{\eqref{Eq:U-01}}{<} \varepsilon.
\end{split}
\]
Given the arbitrariness of $\varepsilon > 0$, \eqref{Eq:2:2} holds. The theorem is proved.
\end{proof}

\subsection{On $U$-sets for convergence over cubes}
\label{p-2-3}
Let $K= ( k_i \in \mathbb{N}, \, i \in \mathbb{N})$ be some fixed sequence. Consider the set
\begin{equation} \label{RD_m(K)}
    WD^{d-m}(2^K  \mathbf{1}) \times \mathbb{G}^{m},
\end{equation}
where $WD^{d-m}(2^K  \mathbf{1})$ is a Dirichlet set (see \eqref{Eq:U-05}) defined for a sequence of the form \( (2^{k_i}\mathbf{1}, \, i \in \mathbb{N}) \).

The following theorems \ref{th-3-1} and \ref{th-3-5} generalize in different directions Theorem 11 from the paper \cite{plotnikov-2010}, where the result was established for the set $WD^{d}(2^K \mathbf{1})$. Also, Theorem \ref{th-3-5} is an analogue of Theorem 4.5 from \cite{kazakova-plotnikov-2025-3}.
\begin{theorem}
\label{th-3-1} The set $WD^{d-m}(2^K  \mathbf{1}) \times \mathbb{G}^{m}$, defined by formula \eqref{RD_m(K)}, is a $U$-set for the $d$-dimensional Walsh system under convergence over cubes.
\end{theorem}
\begin{proof}
    Suppose that there exists a non-identically zero series $(S)$ converging over cubes to zero outside the set $WD^{d-m}(2^K  \mathbf{1}) \times \mathbb{G}^{m}$. Then the quasimeasure $\tau$ generated by the series $(S)$ is not identically zero. Let $\tau$ be concentrated on a dyadic cube $\Delta$ of rank $k_0$. Without loss of generality,
\[
\tau(\Delta) = C >0. 
\]
By Lemma \ref{lem-5}, there exists a point $\boldsymbol{\xi} \in \mathbb{G}^{m}$ such that
$$ \tau(\Delta^* \times \Delta^{(k)}(\boldsymbol{\xi})) > \frac{1}{2^{m(k-k_0)}} C,$$
where $\Delta^*$ --- designation from \eqref{Delta*}.

Since in our case the conditions \eqref{Eq:1:1} and \eqref{Eq:1:11} of Theorem \ref{th-2-1} are satisfied, applying this theorem we obtain
\begin{equation} \label{eq-3-1}
    \lim_{j \rightarrow \infty}2^{mk_j} \int\limits_{\Delta^{*} \times \Delta^{(k_j)}(\boldsymbol{\xi})} R_{k_j {\bf 1}} ({\bf g}^{*}) d\tau \rightarrow 0.
\end{equation}

For ${\bf g} \in WD^{d-m}(2^K  \mathbf{1}) \times \mathbb{G}^{m}$ we have 
\begin{equation} \label{R_}
    R_{k_j {\bf 1}} ({\bf g}^{*})  = 1.
\end{equation}

The series $(S)$ converges to zero over cubes on $\mathbb{G}^d\setminus [WD^{d-m}(2^K  \mathbf{1}) \times \mathbb{G}^{m}]$, and therefore according to Sec. \ref{Subsub:F-Tau-05} $\mathrm{supp} \, \tau \subset WD^{d-m}(2^K  \mathbf{1}) \times \mathbb{G}^{m}$. Thus, \eqref{R_} holds for $\mathbf{g} \in \mathrm{supp} \, \tau$, and consequently by Lemma \ref{lem-4}
\[
 \left|2^{mk_j} \int\limits_{\Delta^{*} \times \Delta^{(k_j)}(\boldsymbol{\xi})} R_{k_j {\bf 1}} ({\bf g}^{*}) d\tau  \right|=  \left|2^{mk_j} R_{k_j {\bf 1}} ({\bf x}^{*}) \tau(\Delta^{*} \times \Delta^{(k_j)}(\boldsymbol{\xi}))\right| > 
\]
\[
>2^{mk_j} \frac{C}{2^{m(k_j - k_0)}} = 2^{mk_0} C.
\]
This contradicts formula \eqref{eq-3-1}. The theorem is proved.
\end{proof}

\begin{theorem} \label{th-3-5} Let $\mathbf{N}_i =  N_i {\bf 1}$, and let the sequence $( N_i \in \mathbb{N}, \, i \in \mathbb{N} )$ satisfy the conditions of Theorem \ref{th-2-2}.  
Then the set $WD^d(\mathbf{N})$ defined by formula \eqref{Eq:U-05} (see Fig. \ref{fig:2}b) is a $U$-set for the $d$-dimensional Walsh system under convergence over cubes. 
\end{theorem}
\begin{proof} 
    Suppose that there exists a non-identically zero series $(S)$ converging over cubes to zero outside the set $WD^{d}(\mathbf{N})$. Then the quasimeasure $\tau$ generated by the series $(S)$ is not identically zero. Let $\tau$ be concentrated on a dyadic cube $\Delta$ of rank $k_0$. Without loss of generality,
\[
\tau(\Delta) = C >0. 
\]
By Theorem \ref{th-2-2} 
\begin{equation} \label{eq-3-l}
    \lim_{i \rightarrow \infty}\int\limits_{\Delta} W_{N_i {\bf 1}} ({\bf g}) d\tau \rightarrow 0.
\end{equation}
For ${\bf g} \in WD^{d}(\mathbf{N})$ we have 
\begin{equation} \label{W_}
 W_{ {N_i {\bf 1}}} ( \mathbf{g}) = 1. 
\end{equation}
The series $(S)$ converges to zero over cubes on $\mathbb{G}^d\setminus WD^{d}(\mathbf{N})$, and therefore according to Sec. \ref{Subsub:F-Tau-05} $\mathrm{supp} \, \tau \subset WD^{d}(\mathbf{N})$. Thus, \eqref{W_} holds for $\mathbf{g} \in \mathrm{supp} \, \tau$, and consequently by Lemma \ref{lem-4}
\[ 
\left|\int\limits_{\Delta} 
W_{ N_i {\bf 1}} d \tau \right| = \tau(\Delta) = C>0.
\]
This contradicts formula \eqref{eq-3-l}. The theorem is proved.
\end{proof}

\begin{theorem} \label{th-3-2} The set $D_m$, defined by formula \eqref{D_m}, is a $U$-set for the $d$-dimensional Walsh system under convergence over cubes.
\end{theorem}
\begin{proof} Since $D_{m} \subset D_{d-2}$ for all $0\le m\le d-2$,
it suffices to show that $D_{d-2}$ is a $U$-set; then the smaller set $D_{m}$ is also a $U$-set. On the other hand,
    \[
    D_{d-2} \subset WD^{2}(2^\mathbb{N}  \mathbf{1}) \times \mathbb{G}^{d-2},
    \]
    since for every $k \in \mathbb{N}$ the conditions $g^{j_1}_{k} =g^{j_{2}}_{k}$ and $R_k(g^{j_1}, g^{j_2}) = 1$ are equivalent. It remains to apply Theorem \ref{th-3-1}. The theorem is proved.
\end{proof}

\begin{theorem} \label{th-3-4}
The set $Q_{m,\mathbf{q}}$, defined by formula \eqref{R_m}, is a $U$-set for the $d$-dimensional Walsh system under convergence over cubes.
\end{theorem}
\begin{proof}
    If there exist $j_1$ and $j_2$ such that $q_{j_1}=q_{j_2}$, then
\[
Q_{m,\mathbf{q}}  \subset D_{d-2},
\]
and consequently $Q_{m,\mathbf{q}}$ is a $U$-set, since by Theorem \ref{th-3-2} $D_{d-2}$ is such.
Therefore, in what follows,
\[
0 \le q_1<q_2<\dots<q_{d-m}.
\]
   Suppose that there exists a non-identically zero series $(S)$ converging over cubes to zero outside the set $Q_{m,\mathbf{q}}$. Then the quasimeasure $\tau$ generated by the series $(S)$ is not identically zero. Let $\tau$ be concentrated on a dyadic cube $\Delta$ of rank $k_0$. Without loss of generality, $\Delta$ is so small that it contains exactly one connected component of the set $Q_{m, \mathbf{q}}$ and
\[
\tau(\Delta) = C >0. 
\]
Let $\Delta = \Delta^{*}\times \Delta_{*}$ (see designation \eqref{Delta*}). Next, let $\Delta^{*} = \widetilde{\Delta}^{*}\times \Delta^{j_{d-m}}$, where $\widetilde{\Delta}^{*}$ --- $(d-m-1)$-dimensional cube corresponding to the coordinates $j_1, \ldots, j_{d-m-1}$, and $\Delta^{j_{d-m}}$ is the interval corresponding to the coordinate $j_{d-m}$. For points $\mathbf{g} \in \mathbb{G}^d$, we adopt a similar notation $\widetilde{\mathbf{g}}^{*} = (g^{j_1}, \ldots, g^{j_{d-m-1}})$. 

By Lemma \ref{lem-5} exists a point $ (\xi^{j_{d-m}},\boldsymbol{\xi}) \in \mathbb{G}^{m+1}$ such that,
\begin{equation} \label{Eq:6:-1}
     \tau(\widetilde{\Delta}^* \times \Delta^{(k)}) > \frac{1}{2^{(m+1)(k-k_0)}} C, \qquad\text{where $\Delta^{(k)}: = \Delta^{(k)}(\xi^{j_{d-m}}) \times \Delta^{(k)}(\boldsymbol{\xi})$}.
\end{equation}
The interval corresponding to the $l$-th coordinate in the cube $\widetilde{\Delta}^{*}$ is partitioned into $2^{k + q_{d-m} - q_l - k_0}$ intervals of rank $k+q_{d-m}-q_l$. Then the cube $\widetilde{\Delta}^{*}$ is partitioned into \[N = 2^{(d-m-1)(k - k_0 + q_{d-m}) - \sum_{l=1}^{d-m-1}q_l}\] dyadic parallelepipeds, which we denote by $P_i$, where $i \in 1:N$, i.e.
\begin{equation}  \label{Eq:6:0} \widetilde{\Delta}^* \times  \Delta^{(k)} = \bigsqcup_{i=1}^{N} P_i \times  \Delta^{(k)}.\end{equation} 
Consider the $i$-th parallelepiped. Two cases are possible.

1) $(P_i \times  \Delta^{(k)}) \bigcap \mathrm{supp} \, \tau = \emptyset$, then \[\tau(P_i \times \Delta^{(k)}) = 0.\]

2) There is $\mathbf{g} \in (P_i \times \Delta^{(k)}) \bigcap \mathrm{supp} \, \tau$. Then, taking into account \eqref{r_2}, for every $k< d-m$ the following holds
\begin{equation}  \label{Eq:6:10}
g^{j_k} \in  \mathrm{C}_{[\,q_{d-m}-q_k\,]}\bigl(g^{j_{d-m}}\bigr).
\end{equation}
And for $j_{d-m}$ holds 
\begin{equation}  \label{Eq:6:1}
g^{j_{d-m}} \in \Delta^{(k)}(\xi^{j_{d-m}}). 
\end{equation}
Expression \eqref{Eq:6:1} is equivalent to the fact that
\begin{equation} \label{Eq:6:2}
g \in \Delta^{(k+q_{d-m}-q_k)}(\eta) \quad \text{for} \; g \in \mathrm{C}_{[\,q_{d-m}-q_k\,]}(g^{j_{d-m}}),  \; \eta \in  \mathrm{C}_{[\,q_{d-m}-q_k\,]}\bigl(\xi^{j_{d-m}}\bigr)
\end{equation}
such that $g\oplus \eta \in \mathbb{G}_{q_{d-m}-q_k}$.

Therefore, according to \eqref{Eq:6:10} and \eqref{Eq:6:2}
\begin{equation} \label{Eq:6:20}
g^{j_k} \in \Delta^{(k+q_{d-m}-q_k)}(\eta^{j_k}) \quad \text{for} \; \eta^{j_k} \in \mathrm{C}_{[\,q_{d-m}-q_k\,]}\bigl(\xi^{j_{d-m}}\bigr) \; \text{and} \; g^{j_k}\oplus\eta^{j_k} \in \mathbb{G}_{q_{d-m}-q_k}. 
\end{equation}

Let the point $\boldsymbol{\eta}: = (\eta^{j_1}, \ldots, \eta^{j_{d-m-1}}) \in \mathbb{G}^{d-m-1}$.
Denote by
\begin{equation} \label{Eq:6:3}
P(\boldsymbol{\eta}) = \mathop{\times}_{l=1}^{d-m-1} \Delta^{k+q_{d-m} - q_l} (\eta^{j_l}) \text{--- a parallelepiped containing the point $\boldsymbol{\eta}$}.
\end{equation}
Since in the proof we are not considering the entire group $\mathbb{G}^d$, but a dyadic cube $\Delta$ chosen such that the contraction operators involved in the definition of $Q_{m, \mathbf{q}}$, when their values are restricted to $\Delta$, become single-valued, then $\boldsymbol{\eta} \in \mathbb{G}^{d-m-1}$, whose components are defined by formula \eqref{Eq:6:20}, is the same for all $\mathbf{g} \in \Delta \cap \mathrm{supp} \ \tau$. Then for every $\mathbf{g} \in (P_i \times \Delta^{(k)}) \bigcap \mathrm{supp} \, \tau$ we have
\[
\mathbf{g} \in P(\boldsymbol{\eta}) \times \Delta^{(k)} .
\]
Since $P(\boldsymbol{\eta})$ is one of the parallelepipeds in the representation \eqref{Eq:6:0}, then in case 2) we obtain that $P_i = P(\boldsymbol{\eta})$.
Thus, 
\begin{equation*}
(P_i \times \Delta^{(k)})\cap \mathrm{supp} \, \tau \neq \emptyset  \iff  P_i = P(\boldsymbol{\eta}).
\end{equation*}
Considering \eqref{Eq:6:0},
 \begin{equation} \label{Eq:6:41}
\tau(\widetilde{\Delta}^* \times \Delta^{(k)}) = \tau(P(\boldsymbol{\eta}) \times \Delta^{(k)}).
\end{equation}
Since in our case the conditions \eqref{Eq:1:1} and \eqref{Eq:1:11} (see Example \ref{exmpl}) of Theorem \ref{th-2-1} are satisfied, applying this theorem we obtain
\begin{equation} \label{Eq:6:5}
\lim_{k\to\infty} 2^{(m+1)k}\int\limits_{\widetilde{\Delta}^* \times  \Delta^{(k)}} R_{k {\bf 1}}(\widetilde{\mathbf{g}}^{*})\,d\tau = 0.
\end{equation}
On the other hand, the integral from formula \eqref{Eq:6:5} can be represented as the sum of two integrals:
\begin{equation} \label{Eq:6:43}
  \int\limits_{P(\boldsymbol{\eta}) \times  \Delta^{(k)}} R_{k {\bf 1}}(\widetilde{\mathbf{g}}^{*})\,d\tau +  \int\limits_{(\widetilde{\Delta}^* \setminus P(\boldsymbol{\eta})) \times  \Delta^{(k)}} R_{k {\bf 1}}(\widetilde{\mathbf{g}}^{*})\,d\tau,
\end{equation}
where the second integral is zero due to formulas \eqref{Eq:6:0} and \eqref{Eq:6:41}. Furthermore, since each side of the parallelepiped $P(\boldsymbol{\eta})$ is an interval of rank greater than $k$ (due to the assumption that all $q_i$ are distinct), then $R_{k \textbf{1}}(P(\boldsymbol{\eta}))$ is well-defined, and then
\begin{equation} \label{Eq:6:42} 
         \int\limits_{P(\boldsymbol{\eta}) \times \Delta^{(k)}} R_{k {\bf 1}}(\widetilde{\mathbf{g}}^{*})\,d\tau 
        = R_{k \textbf{1}}(P(\boldsymbol{\eta})) \tau(P(\boldsymbol{\eta}) \times \Delta^{(k)}).
\end{equation}
Then from \eqref{Eq:6:-1}, \eqref{Eq:6:41}, \eqref{Eq:6:43} and \eqref{Eq:6:42} we obtain 
\[
\left|2^{(m+1)k}\int\limits_{\widetilde{\Delta}^* \times \Delta^{(k)}} R_{k {\bf 1}}(\widetilde{\mathbf{g}}^{*})\,d\tau\right| = 2^{(m+1)k}   \tau(\widetilde{\Delta}^* \times \Delta^{(k)}) >  2^{(m+1)k_0}C.
\]
This contradicts formula \eqref{Eq:6:5}. The theorem is proved.
\end{proof}

A particular case of Theorem \ref{th-3-4} is Corollary \ref{th-3-6}, which provides sufficient conditions on the sequence $\mathbf{N}_i \in \mathbb{N}^2$ under which the Dirichlet sets $WD^2(\mathbf{N}) \subset \mathbb{G}^2$ and the sets $WD^2(\mathbf{N}) \times \mathbb{G}^{d-2} \subset \mathbb{G}^d$ are $U$-sets for convergence over cubes. Specifically, sequences of the form \eqref{spec-N} with $N_i = 2^{i-1+q}$, where $q\in \mathbb{N}$, are considered. In contrast to Theorem \ref{th-3-5}, in Corollary \ref{th-3-6} the components of the two-dimensional sequence $\mathbf{N}_i$ are not only unequal, but also lie in different dyadic blocks, although they are asymptotically close in order $(\lim\limits_{i \rightarrow \infty} N_i/2^{i-1} < \infty)$, which, for example, is not true for $\mathbf{N}_i$ from formula \eqref{luk}.

\begin{cor} \label{important-WD} \label{th-3-6} Let $N_i = 2^{i-1+q}$ ($q \in \mathbb{N}$ is fixed) and the Dirichlet set $WD^2 (\mathbf{N})$ be defined by formula \eqref{Eq:U-05} with the sequence $\mathbf{N}$ defined by formula \eqref{spec-N}. Then the set $WD^2 (\mathbf{N}) \times \mathbb{G}^{d-2}$ is a $U$-set for the $d$-dimensional Walsh system under convergence over cubes.
\end{cor}

\begin{proof}
    We have  
    \[
    (WD^2(\mathbf{N}) \times \mathbb{G}^{d-2}) = \{\mathbf{g}\in \mathbb{G}^d \colon R_{i-1}(g^1)R_{i-1+q}(g^2) = 1 \; \forall i \in \mathbb{N}\} \times \mathbb{G}^{d-2} 
    =R_{d-2},
    \]
since the condition $R_{i-1}(g^1)R_{i-1+q}(g^2) = 1 \; \forall i \in \mathbb{N}$ is equivalent to the condition $\mathrm{C}_{q}(g^1 \oplus x^1)=g^2 \oplus x^2$. It remains to apply the Theorem\ref{th-3-4}.
\end{proof}

\begin{theorem} \label{th-3-3}
The set $P_m$, defined by formula \eqref{P_m}, is a $U$-set for the $d$-dimensional Walsh system under convergence over cubes.
\end{theorem}
\begin{proof}
Suppose that there exists a non-identically zero series $(S)$ converging over cubes to zero outside the set $P_m$. Then the quasimeasure $\tau$ generated by the series $(S)$ is not identically zero. Let $\tau$ be concentrated on a dyadic cube $\Delta$ of rank $k_0$. Without loss of generality,
\[
\tau(\Delta) = C >0. 
\]
By Lemma \ref{lem-5} exists the point $\boldsymbol{\xi} \in \mathbb{G}^{m}$ such that
\begin{equation} \label{eta-1}
    \tau(\Delta^* \times \Delta^{(k)}(\boldsymbol{\xi})) > \frac{1}{2^{m(k-k_0)}} C,
\end{equation} 
where $\Delta^*$ --- designation from \eqref{Delta*}.

Since in our case the conditions \eqref{Eq:1:1} and \eqref{Eq:1:11} (see Example \ref{exmpl}) of Theorem \ref{th-2-1} are satisfied, applying this theorem we obtain
\begin{equation} \label{eq-3-2}
    \lim_{k \rightarrow \infty}2^{mk} \int\limits_{\Delta^{*} \times \Delta^{(k)}(\boldsymbol{\xi})} R_{k {\bf 1}} ({\bf g}^{*}) d\tau = 0.
\end{equation}
For $\mathbf{g} \in P_m$ holds
\begin{equation} \label{eta}
    {\bf g}^{*}  = \boldsymbol{0}.
\end{equation}
The series $(S)$ converges to zero over cubes on $\mathbb{G}\setminus P_m$, and therefore according to Sec. \ref{Subsub:F-Tau-05} $\mathrm{supp} \, \tau \subset P_m$. Thus, \eqref{eta} holds for $\mathbf{g} \in \mathrm{supp} \, \tau$, and consequently, using Lemma \ref{lem-4} and \eqref{eta-1}, we obtain
\[
 \left|2^{mk} \int\limits_{\Delta^{*} \times \Delta^{(k)}(\boldsymbol{\xi})} R_{k{\bf 1}} ({\bf g}^{*}) d\tau \right| = 2^{mk} \tau(\Delta^{*} \times \Delta^{(k)}(\boldsymbol{\xi})) > 2^{mk_0} C.
\]
This contradicts formula \eqref{eq-3-2}. The theorem is proved.
\end{proof}

Note that all sets considered in Theorems \ref{th-3-1}, \ref{th-3-5}, \ref{th-3-2}, \ref{th-3-4}, \ref{th-3-3} are subgroups, and one can consider cosets with respect to these subgroups.
\begin{cor} \label{no-important} Let $WD^{d-m}(2^K  \mathbf{1}) \times \mathbb{G}^{m}$ be defined by formula \eqref{RD_m(K)}, let $WD^d(\mathbf{N})$ be defined by formula \eqref{Eq:U-05} with the conditions on $\mathbf{N}$ from Theorem \ref{th-3-5}, and let $D_m$, $Q_m$, $P_m$ be defined by formulas \eqref{D_m}, \eqref{R_m}, and \eqref{P_m}, respectively. Then for any element $\mathbf{x} \in \mathbb{G}^d$, the cosets
\begin{equation} \label{super-all}
    [WD^{d-m}(2^K  \mathbf{1}) \times \mathbb{G}^{m}] \oplus \mathbf{x}, \; \; WD^d(\mathbf{N}) \oplus \mathbf{x}, \; \; D_m \oplus  \mathbf{x},\; \; Q_{m,\mathbf{q}} \oplus  \mathbf{x},\; \; P_m \oplus  \mathbf{x}
\end{equation}
are $U$-sets for convergence over cubes.
\end{cor}
\begin{proof} The statement follows directly from Theorems \ref{th-3-1}, \ref{th-3-5}, \ref{th-3-2}, \ref{th-3-4}, \ref{th-3-3} and the fact that for systems of characters of abelian groups, the class of $U$-sets is invariant under shifts on the group (for a proof in the one-dimensional case, see \cite{plotnikov-2020}; it generalizes verbatim to the multidimensional case).

\end{proof}
\begin{cor} \label{important} The dyadic planes defined in \eqref{all} are $U$-sets for convergence over cubes. In particular, all planes ${\mathbf{x} } \times \mathbb{G}^m$, parallel to the coordinate planes, are $U$-sets for convergence over cubes for every $m \le d-1$ and every $\mathbf{x} \in \mathbb{G}^{d-m}$.
\end{cor}
Corollary \ref{important} contains the answer to the question of S.F.~Lukomskii from \cite{lukomskii-2021}, formulated in the introduction. This result generalizes both Theorem 3 from \cite{plotnikov-2017} from the case of $\lambda$-convergence to the case of convergence over cubes, and Theorem 1 from \cite{lukomskii-2021} from the case $m=1$ to the case of an arbitrary $m \le d-1$.

\begin{theorem} \label{th-3-7}
    Let $A^1, \ldots, A^q$ --- a finite collection of subsets of the sets from \eqref{super-all} such that $A^i \cap A^j = \emptyset$ for all $i, j \in 1:q$. Then the set 
    \[
    A = A^1 \sqcup \ldots \sqcup A^q
    \]
is a $U$-set for the $d$-dimensional Walsh system under convergence over cubes.
\end{theorem}
\begin{proof}
In Theorems \ref{th-3-1}, \ref{th-3-5}, \ref{th-3-2}, \ref{th-3-4}, and \ref{th-3-3}, we used "information" about the a.e. convergence of the series $(S)$ only on some dyadic cube of rank $k_0$. And since the sets $A^1, \ldots, A^q$ are disjoint, this cube can be made so small that it intersects only one of them. Thus, transitioning to a finite number of disjoint sets does not change the course of the reasoning. The theorem is proved.
\end{proof}
\begin{remark} An arbitrary finite collection of pairwise disjoint dyadic planes is a $U$-set for convergence over cubes. Note that pairwise disjointness of sets is understood in the sense of the group $\mathbb{G}$, which is a more stringent requirement than their pairwise disjointness on the interval. For example, the diagonal $\{\mathbf{g} \in \mathbb{G}^2: g^1 = g^2\}$ and the anti-diagonal $\{\mathbf{g} \in \mathbb{G}^2: g^1 = g^2 \oplus\mathbf{1}\}$ are disjoint in $\mathbb{G}^2$, while their analogs in $[0,1]^2$ intersect at the point $(\frac{1}{2}, \frac{1}{2})$.
\end{remark}
\begin{remark} In some cases, the condition that the dyadic planes are pairwise disjoint can be removed. Thus, for example, for $\mathbf{q}^i \in \mathbb{N}^{d-m}$, $i \in 1:N$ such that all components of the $(d-m)$-dimensional vector of natural numbers are distinct, a set of the form \begin{equation} \label{dly-dok}
    P_{m+1} \cup Q_{m, \mathbf{q}^1} \cup \ldots \cup Q_{m, \mathbf{q}^N}, \quad \text{where} \; P_{m+1} \cap Q_{m, \mathbf{q}^1} \cap \ldots \cap Q_{m, \mathbf{q}^N} = \mathbf{0},
\end{equation} is a $U$-set for convergence over cubes. For the proof, we consider two cases. In the first case, the point $(\xi^{j_{d-m}},\boldsymbol{\xi}) \neq \mathbf{0} \in \mathbb{G}^{m+1}$ (see \eqref{Eq:6:-1}). Then the cube $\Delta$ can be chosen such that inside $\Delta$ the dyadic planes do not intersect. In the second case, $(\xi^{j_{d-m}},\boldsymbol{\xi}) = \mathbf{0}$ and then for each plane $Q_{m, \mathbf{q}^i}$ in formula \eqref{Eq:6:3}, one can assume that the parallelepipeds $P(\boldsymbol{\eta})$ have the form \begin{equation*} 
P^i := \mathop{\times}_{l=1}^{d-m-1} \mathbb{G}_{k+q_{d-m}^i - q^i_l}.
\end{equation*} Note that the restriction of the plane $P_{m+1}$ to the integration domain of the form \eqref{eq-3-2} is entirely contained in $P^i \times \mathbb{G}_k^{m+1}$ for every $i$. Then, repeating the reasoning of Theorem \ref{th-3-4}, with $P(\boldsymbol{\eta})$ of the form
    \[     \mathop{\times}_{l=1}^{d-m-1} \mathbb{G}_{k+\min\limits_{i\in 1:N} [q_{d-m}^i - q^i_l]}\]
we obtain that \eqref{dly-dok} is a $U$-set for convergence over cubes.

In the general case, the question of the necessity of the pairwise disjointness condition remains open. For example, it is not clear whether a two-dimensional "cross"
\begin{equation} \label{krest}
    \{\mathbf{g} \in \mathbb{G}^2: g^1 = \boldsymbol{\eta}, g^2 \in \mathbb{G}\} \cup \{\mathbf{g} \in \mathbb{G}^2: g^1 \in \mathbb{G}, g^2  = \boldsymbol{\xi}\}
\end{equation}
is a $U$-set for convergence over cubes.

Recall that \eqref{krest} is a $U$-set for convergence over rectangles according to \cite{Scv-1975}.
\end{remark}

\subsection{A Cantor–Lebesgue type theorem and a counterexample}
\label{p-2-1}
The following theorem is a generalization of Theorem 1 from \cite{plotnikov-2010}.
\begin{theorem} \label{theorem-1-1}
If the series $(S)$ of the form \eqref{series} converges over cubes to a finite sum on a set $A \subset \mathbb{G}^d$ of positive measure, then for every sequence $(n_s)$ such that the number of nonzero coefficients in the expansions of $n_s$ is bounded
\begin{equation} \label{Eq:1:-1}
    \varlimsup_{s \rightarrow \infty} \# n_s < \infty,
\end{equation}
the following equality holds
\begin{equation} \label{Eq:1:0}
    \lim_{s\rightarrow \infty} c_{n_s{\bf 1}} = 0.
\end{equation}

\end{theorem}

\begin{proof} Let $l_s := \# n_s$ и $n_s = 2^{k_1}+ \ldots 2^{k_{l_s}}$.
Take in the lemma \ref{lem-1} $M_1 =M_2 = n_s$, then we have
\begin{equation*} 
\sum_{\boldsymbol{\sigma}^1 \in \Sigma^d_2}  \ldots \sum_{\boldsymbol{\sigma}^{l_s} \in \Sigma^d_2} \left[ S_{n_s+1}\left( \mathbf{g} \oplus \left( 
\bigoplus_{j=1}^{l_s}
\mathbf{e}^{\boldsymbol{\sigma}^j}_{k_j+1} \right) \right) - S_{n_s}\left( \mathbf{g} \oplus \left( 
\bigoplus_{j=1}^{l_s}
\mathbf{e}^{\boldsymbol{\sigma}^j}_{k_j+1} \right) \right)  \right]
\end{equation*}
\begin{equation} \label{Eq:1:1-1}
    = 2^{l_s(d-1)} c_{n_s\mathbf{1}} W_{n_s\mathbf{1}}(\mathbf{g}).
\end{equation}
On the other hand, following the proof scheme of Theorem 1 from \cite{plotnikov-2010}, which relies on Egorov's theorem, and supplementing it with Lemma \ref{lem-2}, we obtain the existence of a point $\mathbf{t}_s \in \mathbb{G}^d$ possessing the property that for all sufficiently large $s$
\begin{equation} \label{Eq:1:2-2}
    S_{n_s+1}\left( \mathbf{t}_s \oplus \left( 
\bigoplus_{j=1}^{l_s}
\mathbf{e}^{\boldsymbol{\sigma}^j}_{k_j+1} \right) \right) - S_{n_s}\left( \mathbf{t}_s \oplus \left( 
\bigoplus_{j=1}^{l_s}
\mathbf{e}^{\boldsymbol{\sigma}^j}_{k_j+1} \right) \right) < \varepsilon \; 
\end{equation}
for all ${\boldsymbol{\sigma}}^i \in \Sigma^d$, where $i\in 1:l_s$. Due to the fact that the sequence $l_s$ is bounded above, the application of Lemma \ref{lem-2} is justified.

Then from \eqref{Eq:1:1-1} and \eqref{Eq:1:2-2}:
\begin{align*}
|c_{n_s\mathbf{1}}| = |c_{n_s\mathbf{1}} W_{n_s\mathbf{1}}(\mathbf{t}_s)| < \varepsilon.
\end{align*}
By the arbitrariness of $\varepsilon > 0$, \eqref{Eq:1:0} is valid. The Theorem is proved.
\end{proof}

Condition \eqref{Eq:1:-1} is essential here: in Theorem \ref{theorem-1-2} below, a series is constructed that converges everywhere to a finite sum, for which there exists a sequence of diagonal coefficients (with indices not satisfying \eqref{Eq:1:-1}) whose growth is not majorized by any predetermined sequence. This example is a modification of one of the constructions given in Theorem 1 of the paper \cite{plotnikov-2012}.

\begin{theorem} \label{theorem-1-2}
Let a sequence $(m_s)_{s\in \mathbb{N}} \rightarrow \infty$ and a sequence $(n_s)_{s\in \mathbb{N}}$ be given such that $(n_s)_{i} =1$ (dyadic coefficients of $(n_s)$) for $i=0, \ldots, m_s -1$.
    Consider a sequence of nonzero real numbers $B_s$. Then there exists a double Walsh series
    \[
    (S) = \sum_{\alpha =0}^{\infty}\sum_{\beta = 0}^{\infty} c_{\alpha, \beta} W_{\alpha, \beta} (\mathbf{g})
    \]
    converging over squares everywhere on $\mathbb{G}^2$ to a finite sum, for which
    \[
    \varlimsup_{s \rightarrow \infty} \frac{|c_{n_s \bf{1}}|}{|B_s|} = + \infty.
    \]
\end{theorem}
\begin{proof}
Let the sequence $\{d_s\}_{s=0}^{\infty}$ be such that
\[
\varlimsup_{s \rightarrow \infty} \frac{|d_s|}{|B_s|} = + \infty.
\]
Define the coefficients of the series
\[
c_{\alpha, \beta} = 
\begin{cases} 
d_s, & \text{if } \beta = n_s, \; 2^{ \lfloor \log_2 n_s \rfloor} \leq \alpha \leq \frac{1}{2}(2^{\lfloor \log_2 n_s \rfloor}+n_s-1), \\
-d_s, & \text{if } \beta = n_s, \; \frac{1}{2}(2^{\lfloor \log_2 n_s \rfloor}+n_s+1) \leq \alpha \leq n_s, \\ 
0,    & \text{in other cases.}
\end{cases}
\]
Then for $n_s+1 \leq N \leq n_{s+1}$
\begin{equation} \label{eq-2-1}
    S_N (\mathbf{g}) = \sum_{j=0}^{s} d_jW_{n_j}(g^2)(2D_{\frac{1}{2}(2^{\lfloor \log_2 n_j \rfloor}+n_j+1)}(g^1) - D_{2^{\lfloor \log_2 n_j \rfloor}}(g^1) - D_{n_j+1}(g^1)).
\end{equation}
Consider two cases: $g^1=0$ and $g^1 \neq 0$. When $g^1=0$, using \eqref{Eq:WFS}, we have
\[S_N (0, g^2) = \sum_{j=0}^{s} d_jW_{n_j}(g^2) \left(2\left({\frac{1}{2}(2^{\lfloor \log_2 n_j \rfloor}+n_j+1)}\right) - 2^{\lfloor \log_2 n_j \rfloor} - n_j - 1\right) = 0.\]
When $g^1 \neq 0$, for some $q$ we have $g^1 \notin \Delta^{(q)}(0)$. On the other hand, the smallest index of a nonzero dyadic coefficient of the number $n_j +1$ is $m_j$ and according to \eqref{Eq:Pl12-lemm5}, the expression in brackets in \eqref{eq-2-1} is zero outside $\mathbb{G}_{m_j}$. Since $m_j \rightarrow \infty$, we obtain that in \eqref{eq-2-1} there are exactly $J = \min\{j\colon m_j \ge q\}$ nonzero terms for any $s \ge J$. Thus, the constructed series converges to a finite sum at all points of $\mathbb{G}^d$.
\end{proof}

\section{List of illustrations} 
\label{sp-il}
\begin{figure}[H]
    \centering
    \begin{minipage}{0.3\textwidth}
        \centering
        \includegraphics[width=\textwidth]{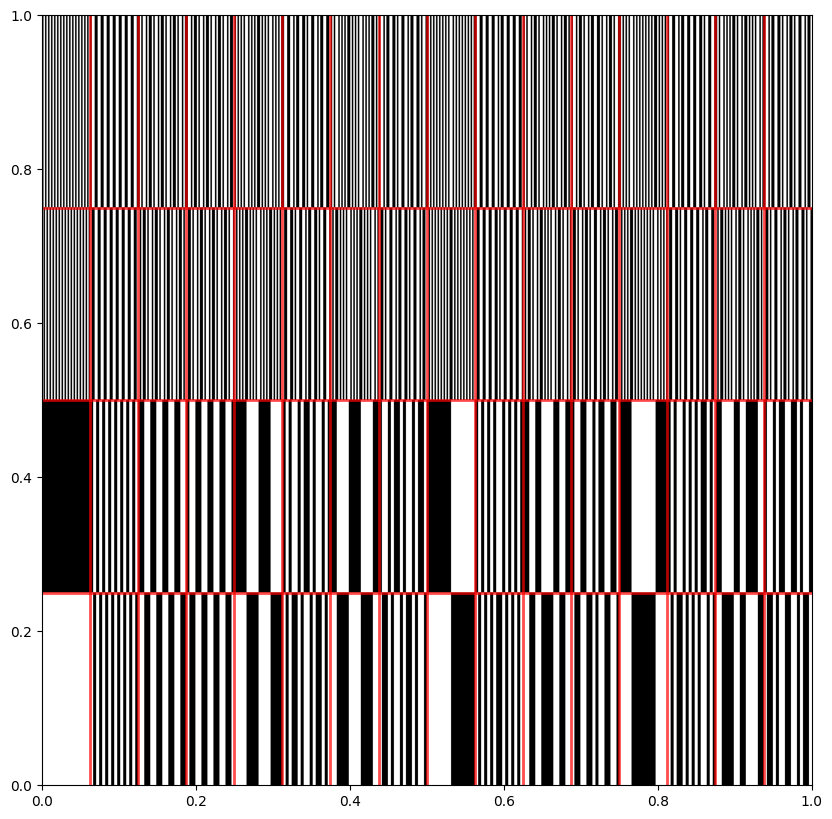}
        \\\small\textbf{(a)} Layer of the Lukomskii set \cite{Lukomskii-1992} $E^2$ --- $M$-set for convergence over cubes
    \end{minipage}
    \hfill
    \begin{minipage}{0.3\textwidth}
        \centering
        \includegraphics[width=\textwidth]{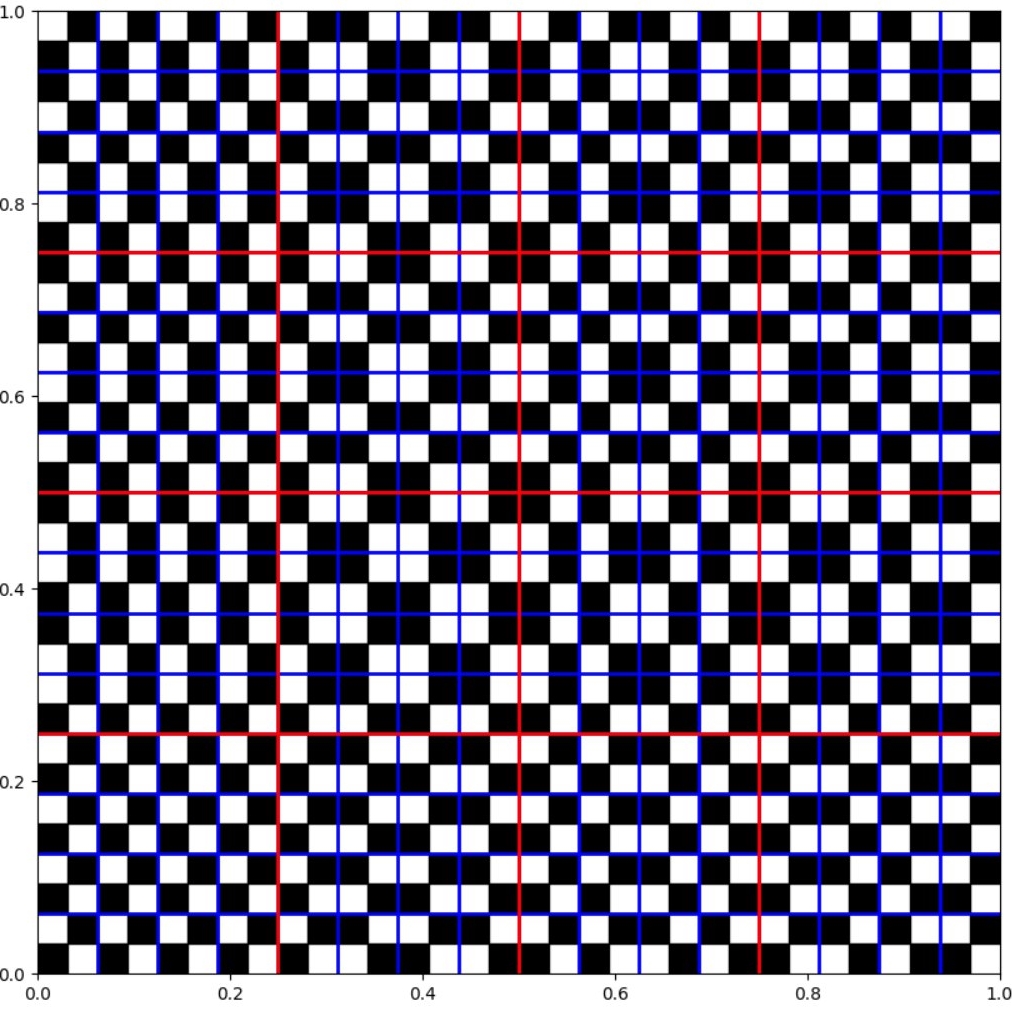}
        \\\small\textbf{(b)} Layer of the set from \cite{kazakova-plotnikov-2025-3} $F^2$ --- an $M$-set for convergence over rectangles
    \end{minipage}
    \caption{Layers of $M$-sets}
    \label{fig:1}
\end{figure}

\begin{figure}[H]
    \centering
    \begin{minipage}{0.3\textwidth}
        \centering
        \includegraphics[width=\textwidth]{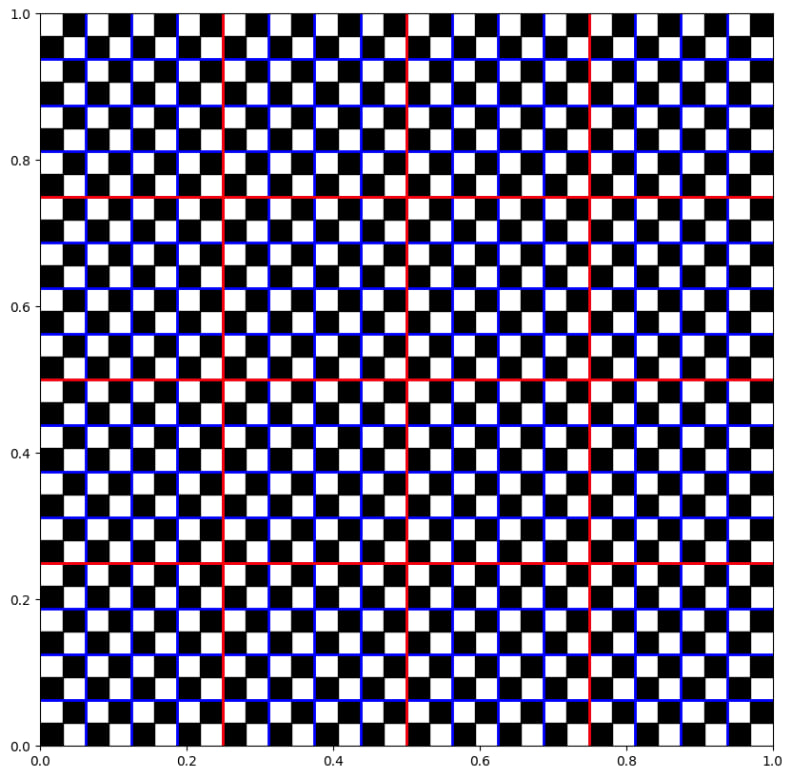}
        \\\small\textbf{(a)} 
    \end{minipage}
    \hfill
    \begin{minipage}{0.3\textwidth}
        \centering
        \includegraphics[width=\textwidth]{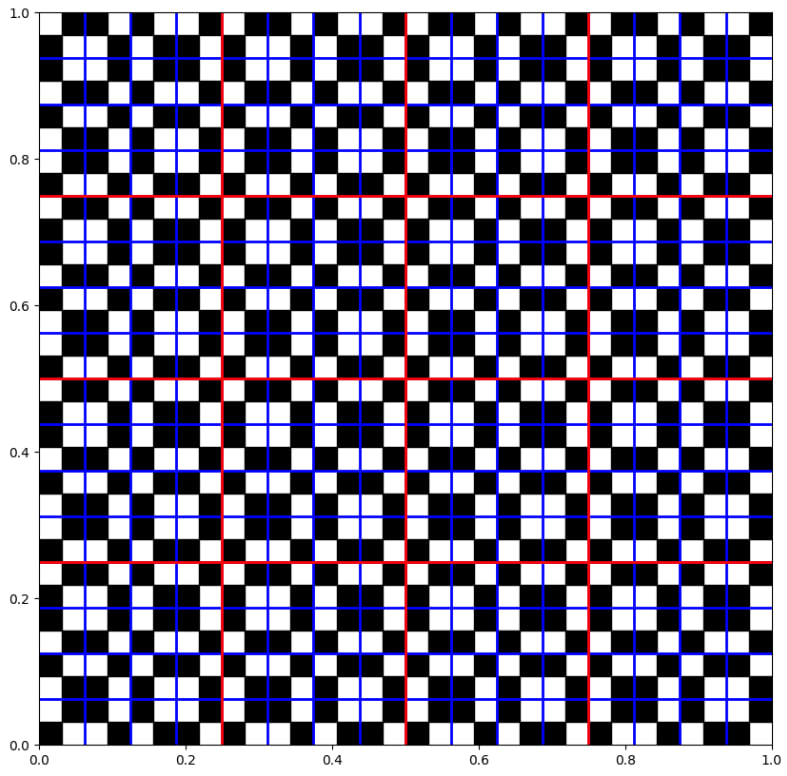}
        \\\small\textbf{(b)}  
    \end{minipage}
    \caption{Layers of Dirichlet-type sets}
    \label{fig:2}
\end{figure}

\begin{figure}[H]
    \centering
    \begin{minipage}{0.22\textwidth}
        \centering
        \includegraphics[width=\textwidth]{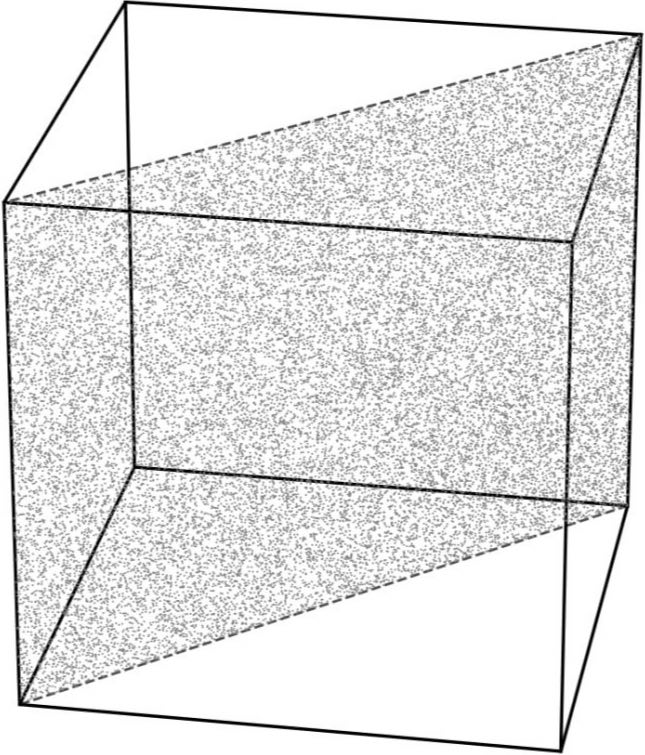}
            \\\small\textbf{(a)} $x=0$
    \end{minipage}
    \hfill
    \begin{minipage}{0.22\textwidth}
        \centering
        \includegraphics[width=\textwidth]{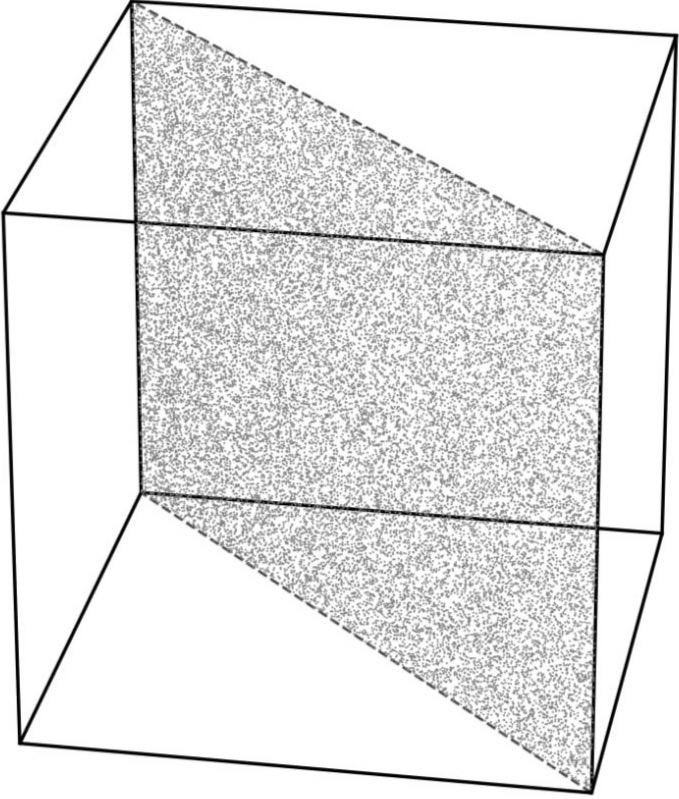}
        \\\small\textbf{(b)} $x = e_0$
    \end{minipage}
    \hfill
    \begin{minipage}{0.22\textwidth}
        \centering
        \includegraphics[width=\textwidth]{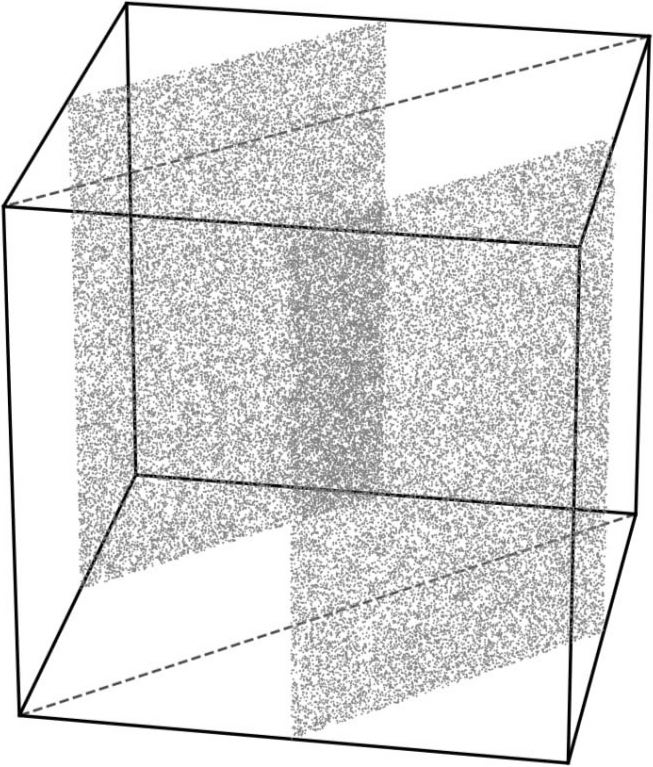}
        \\\small\textbf{(c)} $x = \oplus_{ k = 0 }^\infty e_k$ 
    \end{minipage}
    \hfill
    \begin{minipage}{0.22\textwidth}
        \centering
        \includegraphics[width=\textwidth]{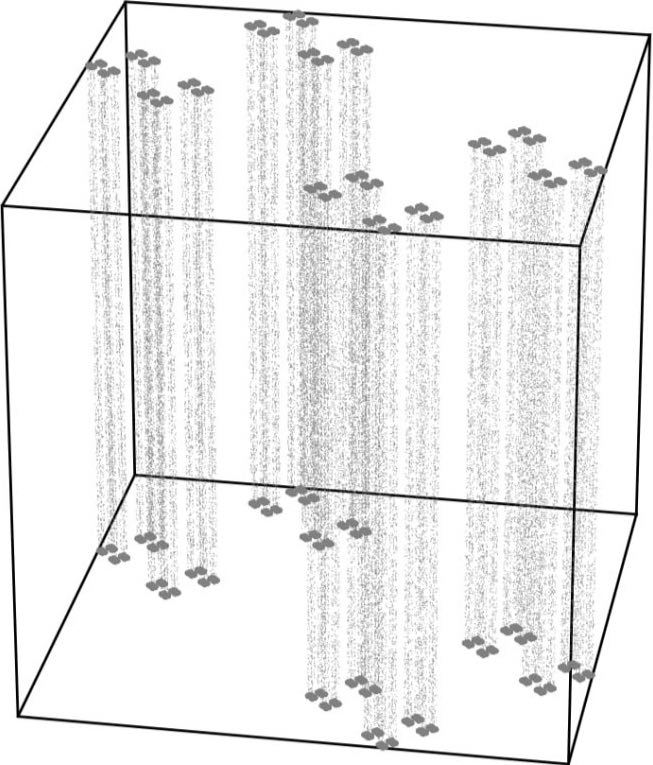}
        \\\small\textbf{(d)}  $x = \oplus_{ k = 0 }^\infty e_{2k+1}$ 
    \end{minipage}
    \caption{Dyadic planes of the form $D_1^{x}$}
    \label{fig:3}
\end{figure}

\begin{figure}[H]
    \centering
    \begin{minipage}{0.22\textwidth}
        \centering
        \includegraphics[width=\textwidth]{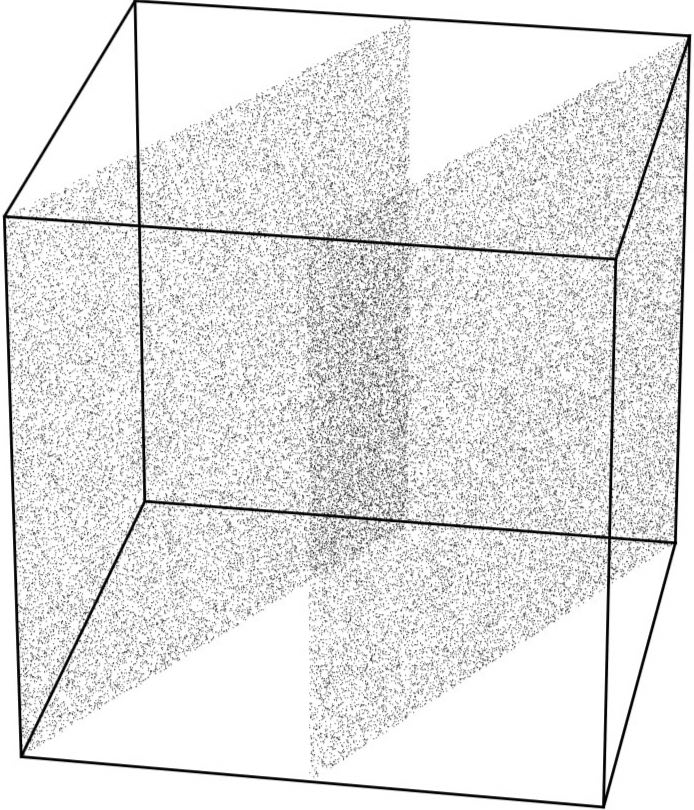}
            \\\small\textbf{(a)} $x=0$
    \end{minipage}
    \hfill
    \begin{minipage}{0.22\textwidth}
        \centering
        \includegraphics[width=\textwidth]{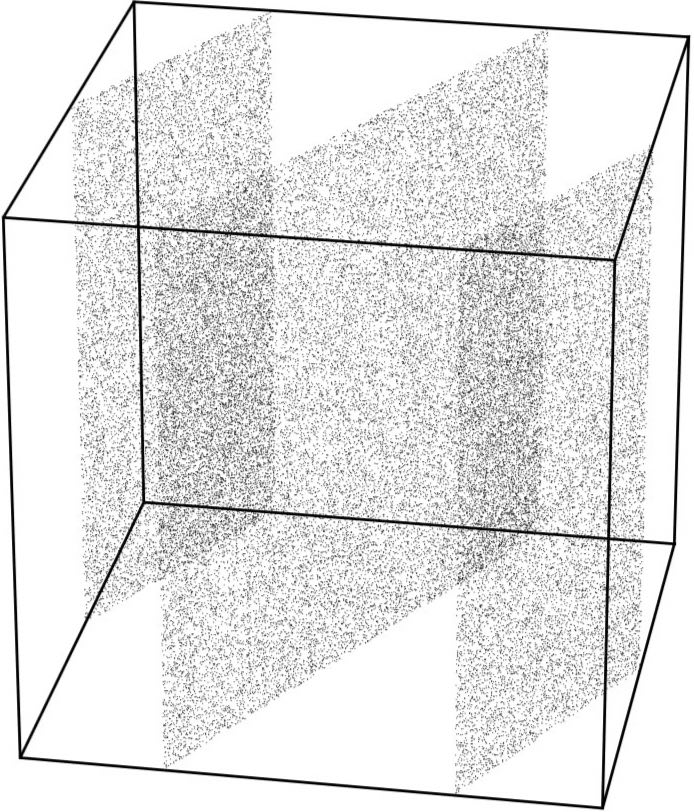}
        \\\small\textbf{(b)} $x = e_0$
    \end{minipage}
    \hfill
    \begin{minipage}{0.22\textwidth}
        \centering
        \includegraphics[width=\textwidth]{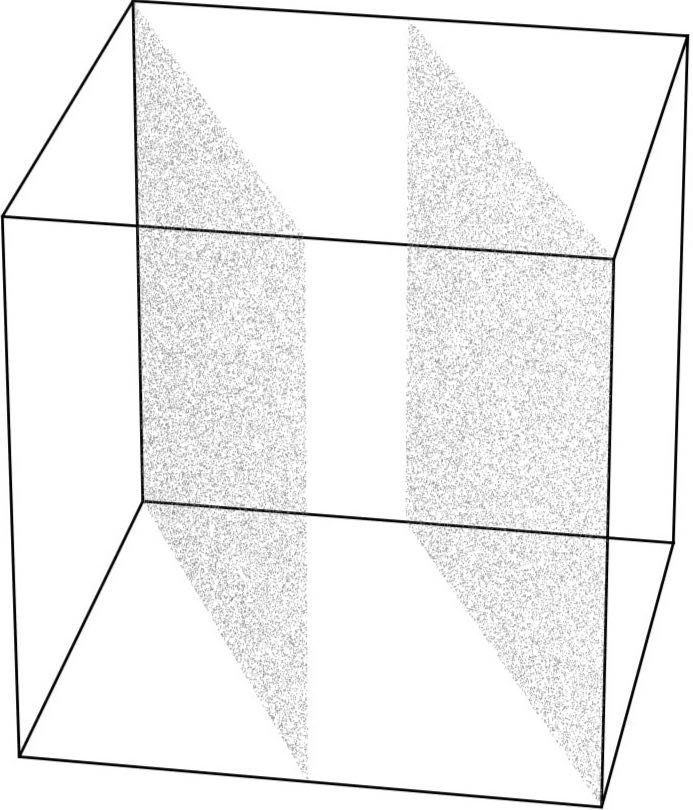}
        \\\small\textbf{(c)} $x = \oplus_{ k = 0 }^\infty e_k$ 
    \end{minipage}
    \hfill
    \begin{minipage}{0.22\textwidth}
        \centering
        \includegraphics[width=\textwidth]{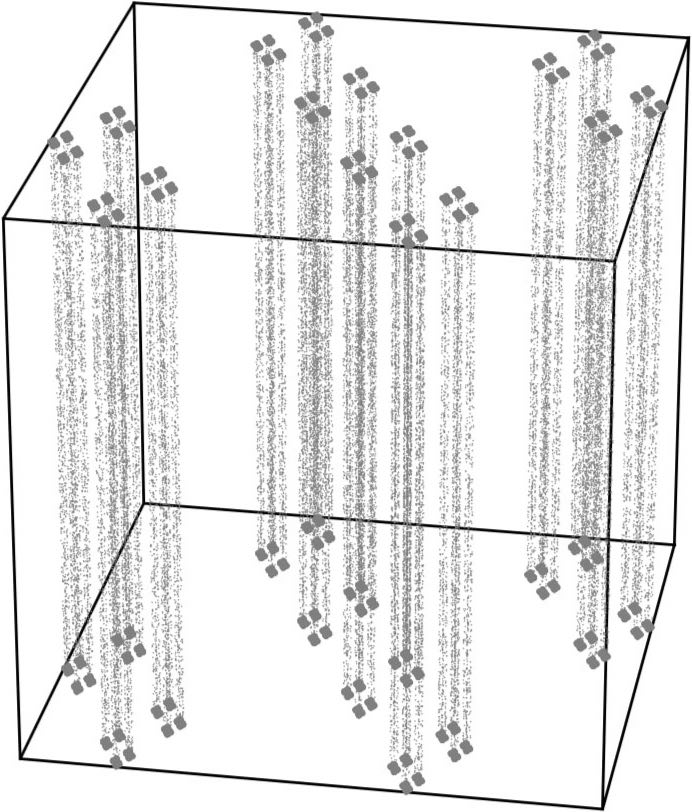}
        \\\small\textbf{(d)}  $x = \oplus_{ k = 0 }^\infty e_{2k+1}$ 
    \end{minipage}
    \caption{Dyadic planes of the form $Q_{1, \mathbf{q}}^{x}$ при $q_1 = 0$ и $q_2 =1$}
    \label{fig:4}
\end{figure}

\begin{figure}[H]
    \centering
    \includegraphics[width=0.22\textwidth]{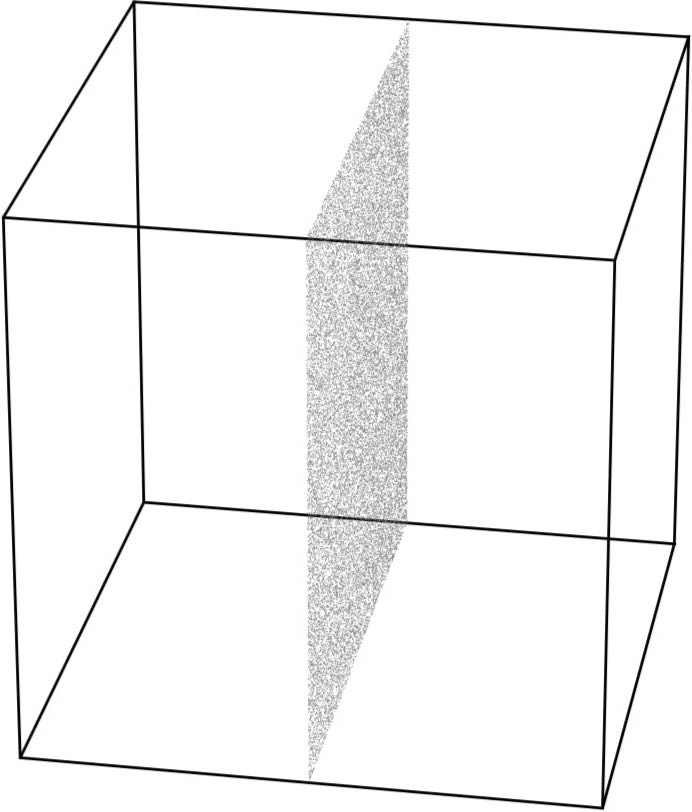}
    \caption{Dyadic plane of the form $P_2^{\mathbf{x}}, \quad \mathbf{x} \in \mathbb{G}^2$}
    \label{fig:5}
\end{figure}

\end{document}